\documentclass[11pt,a4paper]{article}
\usepackage[left=2.0cm, top=2.45cm,bottom=2.45cm,right=2.0cm]{geometry}
\usepackage{mathtools,amssymb,amsthm,mathrsfs,calc,graphicx,xcolor,cleveref,dsfont,tikz,pgfplots,bm,todonotes,calc}
\usepackage[british]{babel}
\usepackage{amsfonts}              
\usepackage[T1]{fontenc}
\usepackage{bbm}
\usepackage{url}

\usepackage{verbatim}

\numberwithin{equation}{section}

\newtheorem {theorem}{Theorem}[section]
\newtheorem {proposition}[theorem]{Proposition}
\newtheorem {lemma}[theorem]{Lemma}

{\theoremstyle{definition}

}
{\theoremstyle{theorem}
\newtheorem {remark}[theorem]{Remark}
\newtheorem {example}[theorem]{Example}
}

\newcommand{\dint}{\textup{d}}

\newcommand{\conv}{\textup{conv}}

\newcommand{\vol}{\mathrm{vol}}

\def\EE{\mathbb{E}}

\def\LL{\mathbb{L}}

\def\NN{\mathbb{N}}

\def\PP{\mathbb{P}}

\def\RR{\mathbb{R}}
\def\SS{\mathbb{S}}

\def\XX{\mathbb{X}}

\makeatletter
\let\@fnsymbol\@alph
\makeatother

\begin{document}

\title{\bfseries Limit laws for longest edges in empty region graphs}

\author{Holger Sambale\footnotemark[1],\;\; Matthias Schulte\footnotemark[2],\;\; and Christoph Th\"ale\footnotemark[3]}


\date{}
\renewcommand{\thefootnote}{\fnsymbol{footnote}}
\footnotetext[1]{Bielefeld University, Germany. Email: hsambale@math.uni-bielefeld.de}

\footnotetext[2]{Hamburg University of Technology, Germany. Email: matthias.schulte@tuhh.de}

\footnotetext[3]{Ruhr University Bochum, Germany. Email: christoph.thaele@rub.de}

\maketitle

\begin{abstract} \noindent Empty region graphs are graphs whose vertices are points in $\mathbb{R}^d$ and where two vertices are connected by an edge whenever some associated region does not contain any other vertices. We investigate the asymptotic behaviour of long edges in empty region graphs generated by a stationary Poisson process in $\mathbb{R}^d$. {Letting} the intensity of the underlying Poisson process tend to infinity, we consider the associated point process of edge midpoints, suitably transformed edge lengths, and directions of the edges. We prove that it converges in distribution to a Poisson process on $\mathbb{R}^d \times \mathbb{R}\times\mathbb{L}^d$, where $\mathbb{L}^d$ is the space of lines in $\mathbb{R}^d$ through the origin, and that the suitably transformed length of the longest edge with midpoint in an observation window converges in distribution to a Gumbel distributed random variable. Our approach yields explicit error bounds in Kantorovich--Rubinstein distance for the point process convergence {when restricting to an observation window} and in Kolmogorov distance for the maximal edge length. The results apply uniformly to a broad class of empty region graphs, including the Gabriel graph, the relative neighbourhood graph, the beta-skeleton graph, the Mastercard graph, and the Pacman graph.
\\[2mm]
\end{abstract}


\section{Introduction and main results}

\subsection{Introductory notes}

Geometric random graphs have been the subject of extensive study over the past decades as they provide a natural mathematical framework for describing spatial networks in stochastic geometry, computational geometry, and applied fields such as telecommunications and biological modelling. Given a random collection of points in Euclidean space, one typically connects two points according to some geometric rule, leading to a wide range of random graph models such as the Gilbert graph, the Delaunay graph, or the Gabriel graph. These models exhibit rich structural properties and have been widely investigated, for example, from the perspective of connectivity, cluster sizes, degree distributions, or total edge lengths.

A prominent class of such random geometric graphs are {so-called} \emph{empty region graphs}, where two points are connected if a prescribed geometric region associated with the pair contains no further points \cite{CardintalColletteLangerman,DevillersEmptyRegion21,MitchellMulzer}. These graphs belong to the broad family of proximity graphs, where the presence of edges is determined by local geometric constraints \cite{MathiesonMoscato}. Examples include the Gabriel graph, the relative neighbourhood graph, and the {beta-skeleton} graph, see Section \ref{sec:ProxGraphs} below for detailed descriptions and illustrations. These and various other constructions arise in algorithmic and computational geometry, morphology, computer vision, geographic analysis, and pattern classification \cite{AryaMount,Jaromczyk_Toussaint,KirkpatrickRadke}. Despite their diverse definitions, empty region graphs share a common structural feature: the presence of long edges is governed by rare events, where unusually large empty regions occur. This makes them a natural setting in which to apply methods from extreme value theory and Poisson approximation to study extremal edge lengths.

In the present work, we consider the asymptotic behaviour of longest edges in empty region graphs generated by a stationary Poisson process in $\RR^d$. Our focus lies on the distribution of the longest edges and on the Poissonian structure emerging in the limit when the edge midpoints and the suitably transformed edge lengths as well as the edge directions are recorded, where the latter are elements in $\LL^d$, the space of all lines through the origin. More precisely, {letting} the intensity of the underlying Poisson process tend to infinity, we show that the associated point process of midpoints, transformed edge lengths, and edge directions converges in distribution to a Poisson process on $\RR^d \times \RR \times\LL^d$. {When restricting to edges with midpoints in an observation window, we obtain explicit error bounds in the Kantorovich--Rubinstein distance and} Gumbel limit laws for the length of the longest edge, together with associated rates of convergence for the Kolmogorov distance.

The study of longest edges in geometric random graphs and related models has a long tradition in stochastic geometry. Early work of Henze~\cite{Henze82} established limit laws for maxima of nearest-neighbour distances, and Penrose~\cite{Penrose97} analysed the asymptotic behaviour of the longest edge in the random minimal spanning tree. More recently, extreme-value methods have been applied to other structures arising in stochastic geometry, such as Poisson--Voronoi or Delaunay tessellations (see, e.g., \cite{CalkaChenavier14,Otto25}) and large $k$th-nearest neighbour balls (see \cite{ChenavierHenzeOtto} and the references therein). In parallel, limit laws for longest edges have also been studied in probabilistic network models, such as long-range percolation~\cite{RousselleSoenmezLRP} and soft random geometric graphs~\cite{RousselleSoenmezSRGG}. Together, these contributions underline the role played by extremal edges in understanding the connectivity and geometry of spatial random graphs. They motivate the investigation of similar questions for empty region graphs, which are in the focus of the present paper. In contrast, cumulative edge length functionals, where all edges of a class of empty region graphs are taken into account, are studied in \cite{GJLR}. In that paper, the authors derive quantitative bounds for the normal approximation of such functionals employing stabilisation arguments.

On a technical level, the quantitative Poisson process limit theorem developed in this paper relies on a very general Poisson process approximation bound from \cite{BSY}. In contrast, the Gumbel limit laws for the length of the longest edge are derived from a new Poisson approximation bound, which is of independent interest. This new bound is obtained by combining arguments from \cite{BSY} with the classical Chen--Stein method for Poisson approximation. 
The main part of the {proofs} then relies on delicate estimates for integrals and on lower bounds for intersection volumes of sets under translations or rigid motions. This task is technically demanding and requires arguments, going beyond those needed in the study of normal approximations for {cumulative edge length functionals.}

\subsection{Empty region graphs}\label{sec:ProxGraphs}

Before we introduce the class of graphs under consideration, let us fix some notation. The volume (i.e., $d$-dimensional Lebesgue measure) of a set $W\subset\RR^d$ is denoted by ${\rm vol}(W)$. For two points $x,y\in\RR^d$ we denote by $[xy]$ the line segment connecting $x$ with $y$ and by $\|x-y\|$ the Euclidean length of $[xy]$. Moreover, we write $\langle x,y\rangle$ for the standard scalar product of $x$ and $y$. For $x\in\RR^d$ and $r>0$ we let $B^d(x,r)$ be the closed ball centred at $x$ with radius $r$. {Furthermore, we introduce the shorthand notation $\kappa_d:={\rm vol}(B^d(0,1))=\pi^{d/2}/\Gamma(d/2+1)$ and put $\SS^{d-1}:=\{x\in\RR^d:\|x\|=1\}$.}

To introduce the class of random graphs we study in this article, let $\eta_t$ be a stationary Poisson process with intensity $t\geq 2$ on $\RR^d$ for some fixed space dimension $d\geq 1$. We remark that the choice of $2$ as a lower bound for $t$ is arbitrary, any number $>1$ could be used instead. With any two distinct points $x,y\in\RR^d$ we associate a compact set $S(x,y)\subset\RR^d$ such that $S(x,y)=S(y,x)$ and
\begin{equation}\label{eq:gamma}
    \vol(S(x,y)) = \gamma\|x-y\|^d
\end{equation}
{for a constant $\gamma\in(0,\infty)$.} Based on the sets $(S(x,y))_{x,y\in\mathbb{R}^d}$, we construct the \textit{empty region graph} $G_S(\eta_t)$  with vertex set $\eta_t$ as follows. We include an edge $[xy]$ between any two different points $x$ and $y$ of $\eta_t$ if and only if the associated set $S(x,y)$ contains no other point of $\eta_t$. Note that $x$ and $y$ may belong to $S(x,y)$ but do not have to.

In this article, we mainly focus on certain types of empty region graphs which we present next. All of them are classical examples of geometric random graphs which have been widely discussed in the literature, mostly in the plane, cf.\ \cite{AryaMount,CardintalColletteLangerman,DevillersEmptyRegion21,Devroye,Jaromczyk_Toussaint,KirkpatrickRadke,MathiesonMoscato,MitchellMulzer}. Simulations of these graphs are shown in Figures \ref{fig:Simulations1}--\ref{fig:Simulations3}, while illustrations of the regions $S(x,y)$ are provided in Figures \ref{fig:Ex123}--\ref{fig:Ex67}.

\begin{figure}[t]
    \centering
    \includegraphics[width=0.30\linewidth]{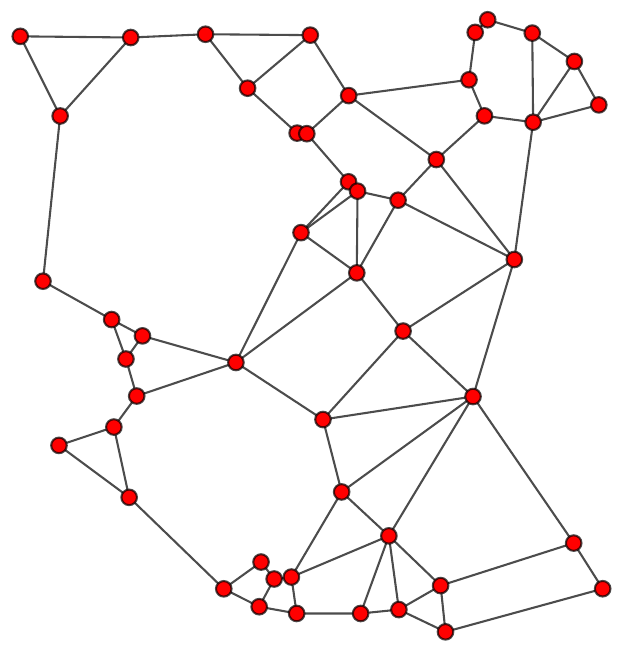}\qquad
    \includegraphics[width=0.30\linewidth]{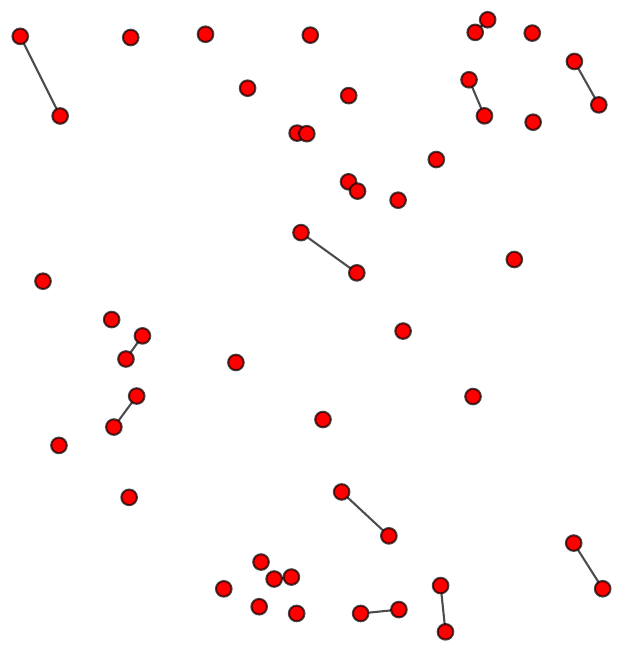}\qquad
    \includegraphics[width=0.30\linewidth]{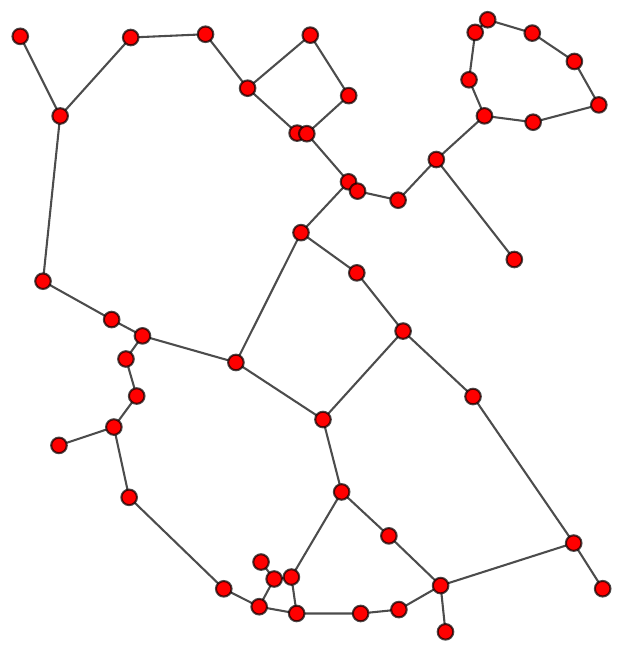}
    \caption{The Gabriel graph (left), the strong nearest neighbour graph (middle), and the relative neighbourhood graph (right) generated by the same set of points.}
    \label{fig:Simulations1}
\end{figure}

\begin{figure}[t]
    \centering
    \includegraphics[width=0.3\linewidth]{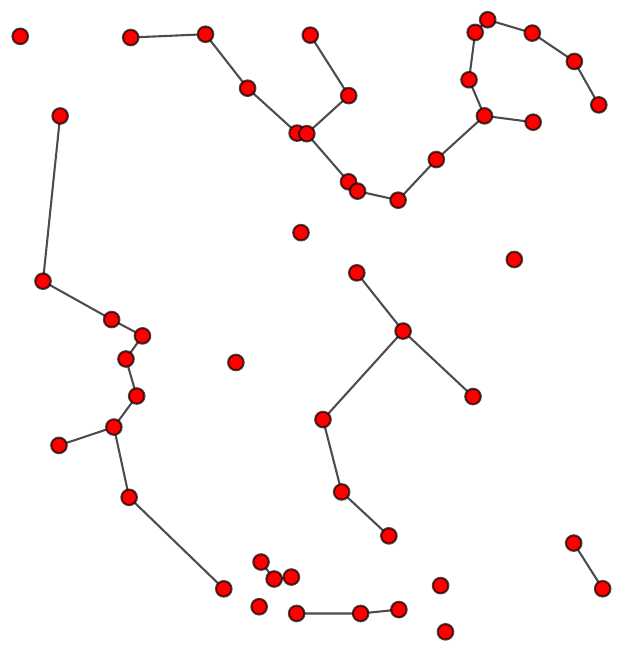}\qquad
    \includegraphics[width=0.3\linewidth]{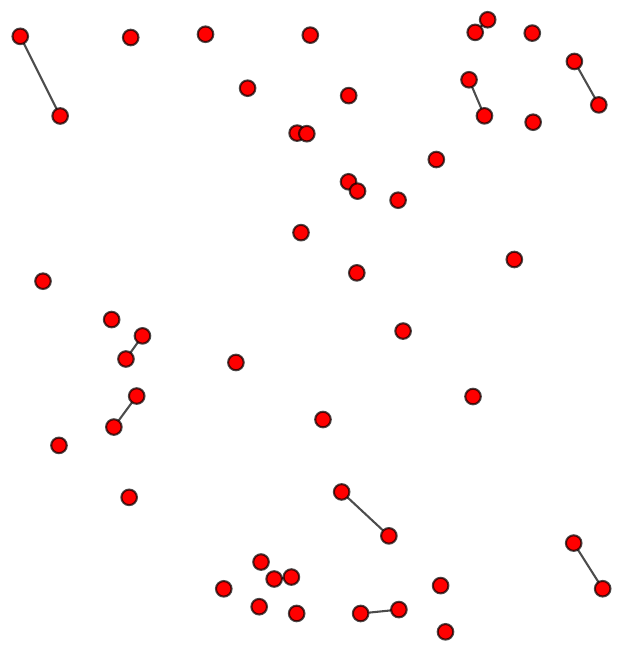}\qquad
    \includegraphics[width=0.3\linewidth]{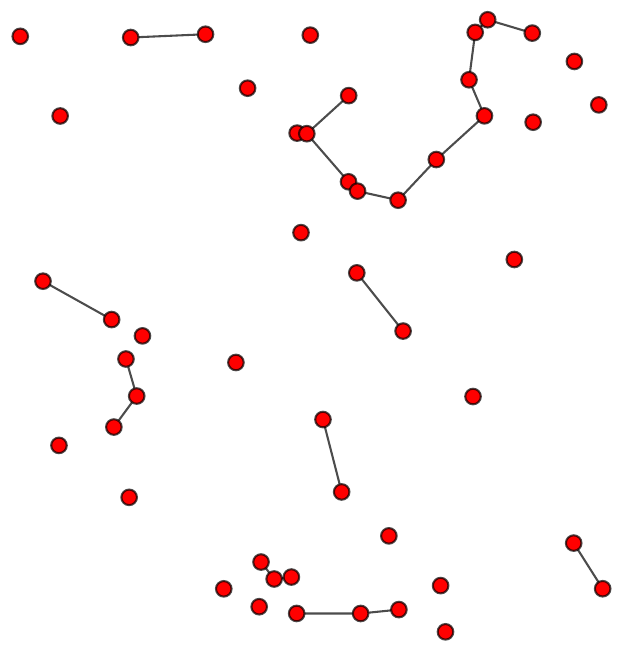}
    \caption{The {beta-skeleton} graph with $\beta=5$ (left), the Mastercard graph (middle), and the truncated slab graph (right) generated by the same set of points.}
    \label{fig:Simulations2}
\end{figure}

\begin{example}[Gabriel graph]\label{ex:Gabriel}\rm
Let $S(x,y):=B^d((x+y)/2,\|x-y\|/2)$ be the $d$-dimensional closed ball whose diameter coincides with the line segment $[xy]$. Then the resulting empty region graph is the classical \textit{Gabriel graph}. In this case, relation \eqref{eq:gamma} is satisfied with $\gamma=\kappa_d/2^d$.
\end{example}

\begin{example}[Strong nearest neighbour graph]\label{ex:StrongNNG}\rm 
Let $S(x,y):=B^d(x,\|x-y\|)\cup B^d(y,\|x-y\|)$ be the union of the two balls with radius $\|x-y\|$ centred at $x$ and $y$, respectively. Then $G_S(\eta_t)$ becomes the \textit{strong nearest neighbour graph}. While in the nearest neighbour graph each vertex is connected to its nearest neighbour, in the strong nearest neighbour graph only edges between so-called mutual nearest neighbours (i.e., one point is the nearest neighbour of the other and vice versa) are present. Here, relation \eqref{eq:gamma} is satisfied with $\gamma={\rm vol}(B^d(0,1)\cup B^d(e_1,1))$, where $e_1:=(1,0,\ldots,0)\in\mathbb{R}^d$.
\end{example}

\begin{example}[Relative neighbourhood graph]\label{ex:RelNNG}\rm 
Suppose that $S(x,y):=B^d(x,\|x-y\|)\cap B^d(y,\|x-y\|)$ is the intersection of two balls with radius $\|x-y\|$ centred at $x$ and $y$, respectively. Then the empty region graph $G_S(\eta_t)$ is the \textit{relative neighbourhood graph} and we have that \eqref{eq:gamma} is satisfied with $\gamma={\rm vol}(B^d(0,1)\cap B^d(e_1,1))$ with $e_1$ as in the previous example. {In this graph, two points are connected by an edge if and only if there is no third point that is closer to both of them than they are to each other.} {The relative neighbourhood graph} is almost surely connected as a consequence of \cite[Theorem 1]{Toussaint1980} and triangle free by \cite[Theorem 17]{CardintalColletteLangerman}.
\end{example}

\begin{example}[{Beta-skeleton graph}]\label{ex:BetaSkel}\rm 
Fix $\beta\in[1,\infty)$ and put
\[
S_\beta(x,y) := B^d\Big(\Big(1-{\beta\over 2}\Big)x+{\beta\over 2}y,{\beta\over 2}\|x-y\|\Big)\cap B^d\Big(\Big(1-{\beta\over 2}\Big)y+{\beta\over 2}x,{\beta\over 2}\|x-y\|\Big)\,.
\]
The corresponding empty region graph $G_{S_\beta}(\eta_t)$ is the {\textit{beta-skeleton graph}} from \cite{KirkpatrickRadke}. The family of these geometric random graphs have important applications to the analysis of interpoint linkages in empirical networks. Note that $G_{S_{\beta_1}}(\eta_t)$ is almost surely a subgraph of $G_{S_{\beta_2}}(\eta_t)$, provided that {$\beta_1\geq\beta_2\ge 1$.} We also notice that $G_{S_1}(\eta_t)$ reduces to the Gabriel graph, $G_{S_2}(\eta_t)$ to the relative neighbourhood graph, and that for general $\beta\in[1,\infty)$ relation \eqref{eq:gamma} is satisfied with
\[
\gamma ={\rm vol}\Big(B^d\Big({\beta\over 2}e_1,{\beta\over 2}\Big)\cap B^d\Big(\Big(1-{\beta\over 2}\Big)e_1,{\beta\over 2}\Big)\Big)
\]
and $e_1$ as in the previous two examples.
\end{example}

\begin{example}[Mastercard graph]\label{ex:Mastercard}\rm 
To construct the so-called \textit{Mastercard graph} from \cite{CardintalColletteLangerman} we take for $S(x,y)$ the convex hull ${\rm conv}\{B^d(x,\|x-y\|)\cup B^d(y,\|x-y\|)\}$. Equivalently, $S(x,y)$ is the Minkowski sum of the segment $[xy]$ with the ball $B^d(0,\|x-y\|)$ of radius $\|x-y\|$. By Steiner's formula from convex geometry see, for example, \cite[Equation (14.5)]{SW} we have $\vol(S(x,y)) = (\kappa_d + \kappa_{d-1})\|x-y\|^d$, implying that \eqref{eq:gamma} is satisfied with $\gamma=\kappa_d + \kappa_{d-1}$.
\end{example}

\begin{example}[Pacman graph]\label{ex:Pacman}\rm 
Fix $\theta\in(0,2\pi]$ and write $c_\theta:=\cos(\theta/2)$. Define the wedge with {apex $x\in\RR^d$, axis $u\in\SS^{d-1}$, angle $\theta$, and radius $r>0$} by
\[
W(x,u,\theta;r):=\{x\}\cup\Big\{z\in\RR^d:0<\|z-x\|\le r,\frac{\langle z-x,\,u\rangle}{\|z-x\|}\ge c_\theta\Big\}.
\]
{For two distinct points $x,y\in\RR^d$ let $u_{xy}:=\frac{y-x}{\|y-x\|}$.} Following \cite{CardintalColletteLangerman}, the Pacman region associated with $x,y$ is
\[
P_\theta(x,y):={\rm conv}\Big\{W(x,u_{xy},\theta;{\|x-y\|})\cup W(y,-u_{xy},\theta;{\|x-y\|})\Big\}
\]
and the \textit{Pacman graph} with parameter $\theta$ arises by taking $S(x,y):=P_\theta(x,y)$.  Then $\vol(S(x,y))=\gamma(d,\theta)\,\|x-y\|^d$ for some constant $\gamma(d,\theta)$ independent of $x,y$, verifying \eqref{eq:gamma}. For $d=2$, if $\theta={2\pi/3}$ we get back the relative neighbourhood graph and for $\theta=2\pi$ we obtain the Mastercard graph considered above. For $d=2$, the Pacman graph with parameter $\theta=4\pi/k$, $k\in\NN$ with $k\geq 2$, has the particular feature that it contains no $k$-star, see \cite[Theorem 16]{CardintalColletteLangerman}. Moreover, the planar Pacman graph with parameter $\theta=\pi$ is free of $4$-cycles, whereas for $\theta=6\pi/5$ the planar Pacman graph is free of $5$-cycles, see \cite[Theorems 18 and 19]{CardintalColletteLangerman}.
\end{example}

\begin{figure}[t]
    \centering
    \includegraphics[width=0.3\linewidth]{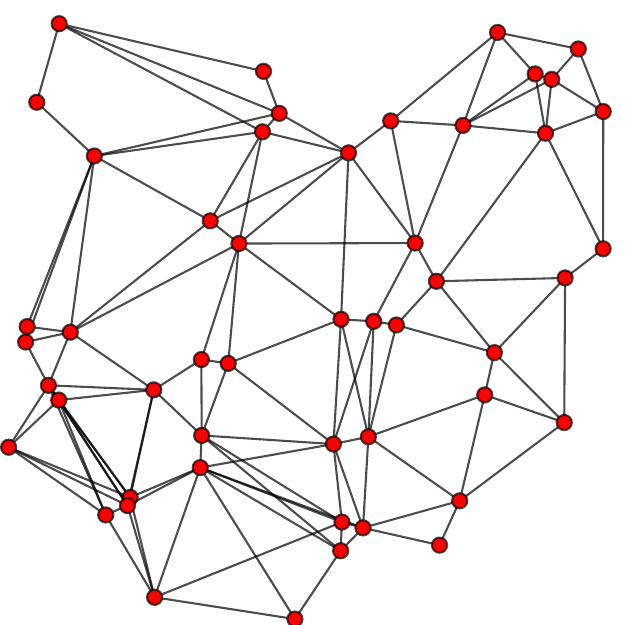}\qquad
    \includegraphics[width=0.3\linewidth]{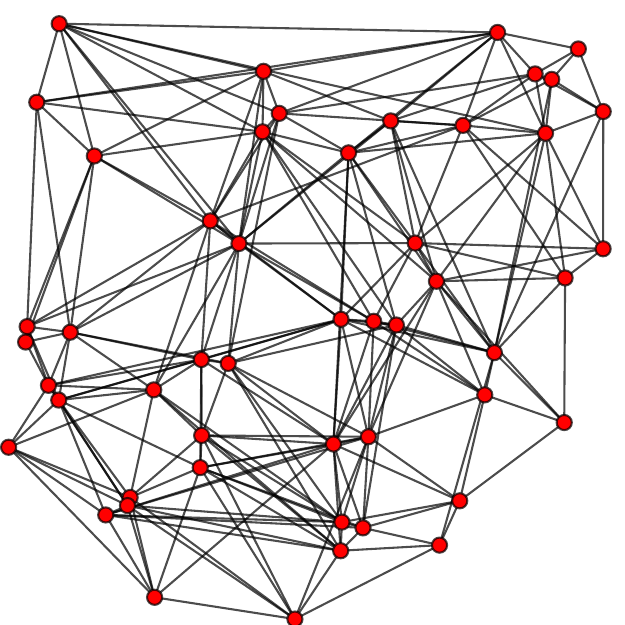}\qquad
    \includegraphics[width=0.3\linewidth]{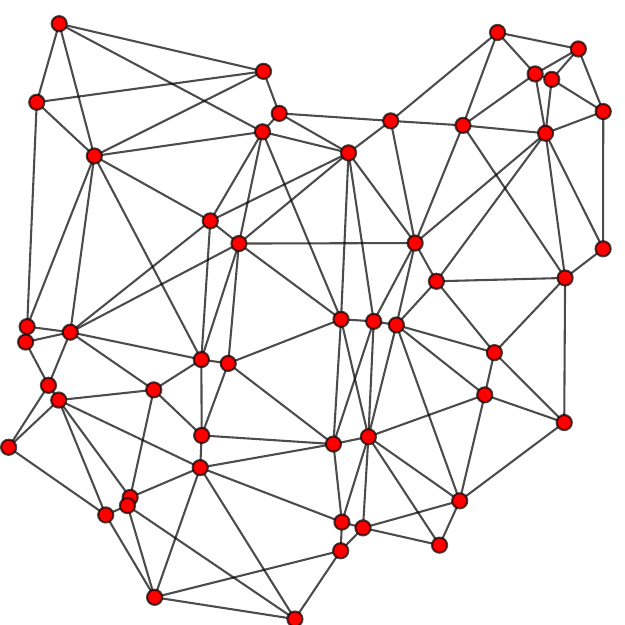}
    \caption{The generalised Gabriel graph with $K=[-1/2,1/2]^2$ (left), $K=\conv\{(-1/2,0),(1/2,0),(0,1/2)\}$ (middle), and $K$ an axis-aligned ellipse with semi axes $1/2$ and $3/10$ (right) generated by the same set of points.}
    \label{fig:Simulations3}
\end{figure}

\begin{example}[Truncated slab graph]\label{ex:TruncatedSlab}\rm 
For two distinct points $x,y\in\RR^d$ define again $u_{xy}:={y-x\over\|y-x\|}$, and let $H_{xy}:=\{z\in\RR^d: |\langle z-{x+y\over 2},{u_{xy}}\rangle|\leq {\|x-y\|\over 2}\}$ and
$$
S(x,y) := H_{xy} \cap B^d(x,2\|x-y\|) \cap B^d(y,2\|x-y\|).
$$
In \cite{CardintalColletteLangerman}, this set is called truncated slab and verbally described as ``two lunes on top of each other separated by a distance $2\,\|x-y\|$'' for $d=2$. Condition \eqref{eq:gamma} is satisfied again with some constant $\gamma$ only depending on $d$. More precisely, $\gamma=\vol(H_0\cap B^d(-{1\over 2}e_1,2))\cap B^d({1\over 2}e_1,2))$, where $e_1=(1,0,\ldots,0)$ and $H_0:=\{z\in\RR^d:|\langle z,e_1\rangle|\leq 1/2\}$. The resulting empty region graph is the so-called \textit{truncated slab graph} from \cite{CardintalColletteLangerman}. We remark that according to \cite[Lemma 20]{CardintalColletteLangerman} it is free of $5$-cycles for $d=2$.
\end{example}

In the Gabriel graph from Example \ref{ex:Gabriel}, the ball around the midpoint between the two vertices can be replaced by a rescaled copy of another fixed set $K$. For example, $K{:=}\{(z_1,z_2)\in\RR^2: |z_1|+|z_2|\leq 1\}$ was considered in \cite{Lee85}. We focus on the following generalised version of the Gabriel graph.

\begin{example}[Generalised Gabriel graph]\label{ex:ConvBody}\rm 
Let $S(x,y)$ be of the form
$$
S(x,y) = { \frac{x+y}{2} + \lVert x-y \rVert K,}
$$
where $K\subset\RR^d$ is a compact set that has non-empty interior and is star-shaped (i.e., $[0z]\subset K$ for all $z\in K$). The resulting empty region graph $G_S(\eta_t)$ is called {\em generalised Gabriel graph}. Here, \eqref{eq:gamma} holds with $\gamma = \vol(K)$.
\end{example}

\begin{figure}[t]\label{fig:Generalised_Gabriel}
\centering
\begin{minipage}{0.20\textwidth}
\centering
\begin{tikzpicture}[scale=1]
\coordinate (x) at (0,0);
\coordinate (y) at (2,0);
\draw[fill=red!10, draw=red] (1,0) circle ({0.5*veclen(2-0,0-0)});
\draw[thick] (x) -- (y);
\fill (x) circle (0.05) node[left] {$x$};
\fill (y) circle (0.05) node[right] {$y$};
\end{tikzpicture}
\end{minipage}%
\hfill
\begin{minipage}{0.39\textwidth}
\centering
\begin{tikzpicture}[scale=1]
\coordinate (x) at (0,0);
\coordinate (y) at (2,0);
\draw[fill=red!10, draw=red] (x) circle (2);
\draw[fill=red!10, draw=red] (y) circle (2);
\draw[red] (x) circle (2);
\draw[red] (y) circle (2);
\draw[thick] (x) -- (y);
\fill (x) circle (0.05) node[left] {$x$};
\fill (y) circle (0.05) node[right] {$y$};
\end{tikzpicture}
\end{minipage}%
\hfill
\begin{minipage}{0.39\textwidth}
\centering
\begin{tikzpicture}[scale=1]
\coordinate (x) at (0,0);
\coordinate (y) at (2,0);
\begin{scope}
  \clip (x) circle (2);
  \fill[red!10] (y) circle (2);
\end{scope}
\draw[red, thick] (x) circle (2);
\draw[red, thick] (y) circle (2);
\draw[thick] (x) -- (y);
\fill (x) circle (0.05) node[left] {$x$};
\fill (y) circle (0.05) node[right] {$y$};
\end{tikzpicture}
\end{minipage}
\caption{Two points $x,y$ and the regions $S(x,y)$ for the Gabriel graph (left), the strong nearest neighbour graph (middle), and the relative neighbourhood graph (right).}
\label{fig:Ex123}
\end{figure}
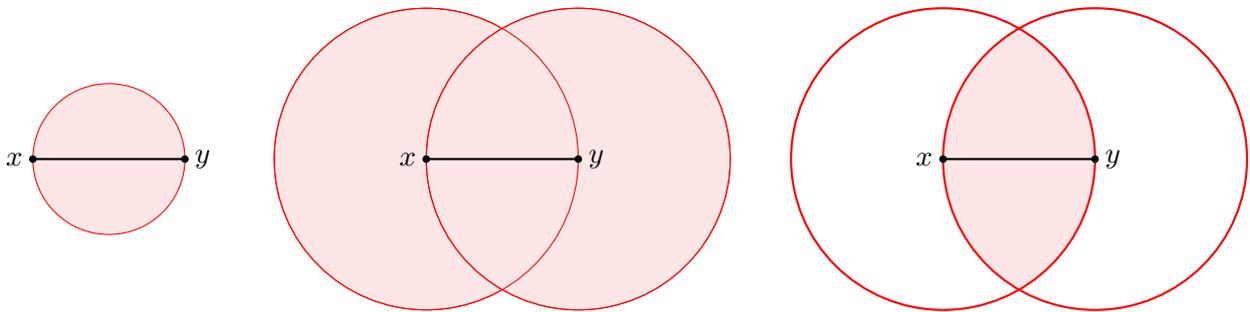

\begin{figure}[t]
\centering
\begin{minipage}{0.59\textwidth}
\centering
\begin{tikzpicture}[scale=0.8]
\coordinate (x) at (0,0);
\coordinate (y) at (2,0);

\coordinate (c1) at (-1,0);
\coordinate (c2) at (3,0);

\begin{scope}
  \clip (c1) circle (3);
  \fill[red!10] (c2) circle (3);
\end{scope}

\draw[red, thick] (c1) circle (3);
\draw[red, thick] (c2) circle (3);

\draw[thick] (x) -- (y);
\fill (x) circle (0.05) node[left] {$x$};
\fill (y) circle (0.05) node[right] {$y$};

\fill (c1) circle (0.03);
\fill (c2) circle (0.03);
\end{tikzpicture}
\end{minipage}%
\hfill
\begin{minipage}{0.39\textwidth}
\centering
\begin{tikzpicture}[scale=1]
\coordinate (x) at (0,0);
\coordinate (y) at (2,0);
\def\r{2}

\fill[red!10]             (0,\r) arc[start angle=90, end angle=270, radius=\r]
                        -- (2,-\r)
                           arc[start angle=270, end angle=90, radius=\r]
                        -- (2,\r) -- cycle;

\draw[fill=red!10,draw=red!10] (2,0) circle ({veclen(2-0,0-0)});
\draw[red, thick] (0,0) circle (2);
\draw[red, thick] (2,0) circle (2);

\draw[thick,red] (0,-2) -- (2,-2);
\draw[thick,red] (0,2) -- (2,2);

\draw[thick] (x) -- (y);
\fill (x) circle (0.05) node[left] {$x$};
\fill (y) circle (0.05) node[right] {$y$};
\end{tikzpicture}
\end{minipage}%
\caption{Two points $x,y$ and the regions $S(x,y)$ for the beta-skeleton graph with $\beta=3$ (left) and the Mastercard graph (right).}
\label{fig:Ex45}
\end{figure}
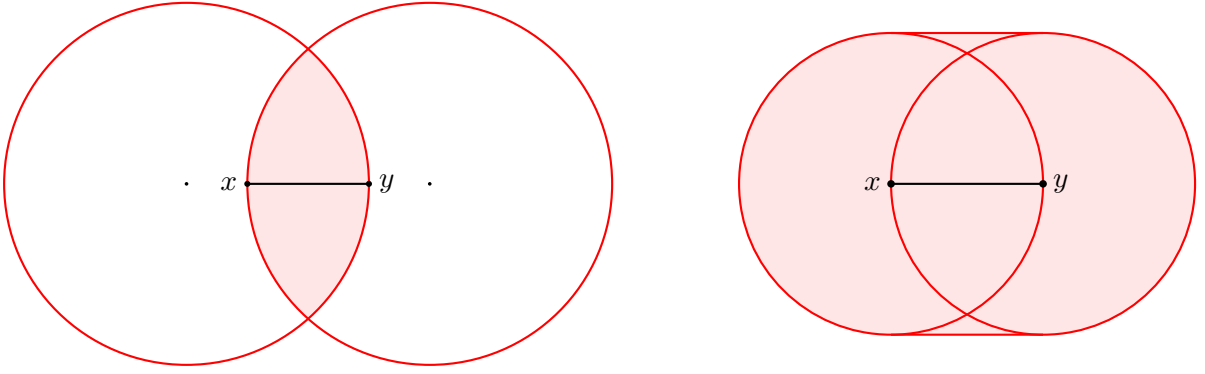

\subsection{Main results}

All results in this section are stated for the empty region graphs introduced in Examples \ref{ex:Gabriel}--\ref{ex:ConvBody} of the previous section. Nevertheless, our approach applies to a broader class of empty region graphs. Since the formulation of such general results requires several technical assumptions, we present the general versions of our results in Theorems \ref{thm:PoisConvAbstract} and \ref{thm:GumbelAbstract} below, while here we restrict attention to Examples \ref{ex:Gabriel}--\ref{ex:ConvBody}.

To formulate our first main result, we need some more concepts and notation. To introduce them, let $\mathbb{X}$ be a locally finite second countable Hausdorff space with Borel $\sigma$-field $\mathcal{B}(\mathbb{X})$ and let $\mathbf{N}_{\mathbb{X}}$ denote the space of locally finite counting measures on $\mathbb{X}$. In our paper, $\XX$ will be the product of $\RR^d$ with $\RR$ and the space $\LL^d$ of $1$-dimensional linear subspaces of $\RR^d$. A random element in $\mathbf{N}_{\mathbb{X}}$ is called a point process. In particular, by a Poisson process $\Pi$ on $\mathbb{X}$ with a locally finite intensity measure $\mathbf{M}$ we understand a point process with the two properties that $\Pi(A)$ follows a Poisson distribution with parameter $\mathbf{M}(A)$ for all $A\in\mathcal{B}(\mathbb{X})$ and that $\eta(A_1),\hdots,\eta(A_m)$ are independent for all pairwise disjoint $A_1,\hdots,A_m\in\mathcal{B}(\mathbb{X})$ with $m\in\mathbb{N}$, see e.g.\ \cite{LP,SW}.

To measure the distance between two point processes $\xi$ and $\zeta$ on $\mathbb{X}$ such that $\mathbb{E}[\xi(\mathbb{X})],\mathbb{E}[\zeta(\mathbb{X})]<\infty$,
we use the \emph{Kantorovich--Rubinstein distance}. It is defined as
\[
\mathrm{d}_{\mathrm{KR}}(\xi,\zeta) :=
\sup\Bigl\{\,\bigl|\mathbb{E}[h(\xi)] - \mathbb{E}[h(\zeta)]\bigr| 
:\; h:\mathbf{N}_{\mathbb{X}} \to \RR,\; h \text{ is 1-Lipschitz w.r.t.\ } \mathrm{d}_{\rm TV}\Bigr\},
\]
where $\mathrm{d}_{\rm TV}$ denotes the {\em total variation distance} which for any two measures $\varphi_1,\varphi_2$ on $\mathbb{X}$ is given by
\begin{equation}\label{eqn:dTV_measures}
\mathrm{d}_{\rm TV}(\varphi_1,\varphi_2) := \sup_{B\in\mathcal{B}(\mathbb{X}),\varphi_1(B),\varphi_2(B)<\infty}
\bigl|\varphi_1(B) - \varphi_2(B)\bigr|.
\end{equation}
The Kantorovich--Rubinstein distance is also called Wasserstein distance, Monge--Kantorovich distance or Rubinstein distance, see the discussion in Section 2 of \cite{DST}. In particular, convergence in Kantorovich--Rubinstein distance implies distributional convergence of point process, a notion for which we refer to \cite[Chapter 23]{Kallenberg}.

In what follows, we write {$\mathrm{span}\,B$} for the linear hull of a set $B\subset\RR^d$. By $\nu$ we denote the rotation-invariant Haar probability measure on $\LL^d$. Recall that $\eta_t$ stands for a Poisson process on $\RR^d$ whose intensity measure measure is $t\geq 2$ times the Lebesgue measure. We write $\eta_{t,\neq}^2$ for the set of pairs of distinct points of $\eta_t$ and $\delta_z$ for the Dirac measure concentrated at a point $z$. In each of the Examples \ref{ex:Gabriel}--\ref{ex:ConvBody}, one has an edge between $x$ and $y$ with $(x,y)\in\eta^2_{t,\neq}$ if and only if the set $S(x,y)$ does not contain any other points of $\eta_t$. The latter can be formulated as $(\eta_t-\delta_x-\delta_y)(S(x,y))=0$ since $\eta_t-\delta_x-\delta_y$ is the Poisson process $\eta_t$ without the points $x$ and $y$. Recall that $\gamma$ is the constant from \eqref{eq:gamma}. We consider the point process
\[
\xi_t := \tfrac{1}{2} \sum_{(x,y) \in \eta^2_{t,\ne}} 
\mathbbm{1}\Big\{ (\eta_t - \delta_x - \delta_y)\big(S(x,y)\big) = 0 \Big\} 
\, \delta_{\big(\tfrac{x+y}{2},\, \gamma t \lVert x - y\rVert^d - \log(t) - \log(\kappa_d/(2\gamma)), \mathrm{span} \{x-y\}\big)},
\]
which records the midpoints, the transformed lengths, and the directions of edges in the empty region graph. It is divided by two since each pair of points occurs twice in the sum. The following theorem shows that the distribution of $\xi_t$ approaches that of a Poisson process on $\RR^d \times \RR \times \LL^d$ with intensity measure $\vol \otimes \mu \otimes \nu$ as $t\to\infty$, where, here and throughout this paper, $\mu$ denotes the measure on $\RR$ defined by $\mu([a,\infty)) = e^{-a}$ for $a\in\RR$. In particular, we obtain explicit dimension-dependent convergence rates in the Kantorovich--Rubinstein distance {after restricting to an observation window.} In what follows, $\varphi|_A$ stands for the restriction of a point process $\varphi$ to a set $A$ (i.e., the point process containing only the points of $\varphi$ that belong to $A$) and $\overset{d}{\longrightarrow}$ indicates convergence in distribution. 

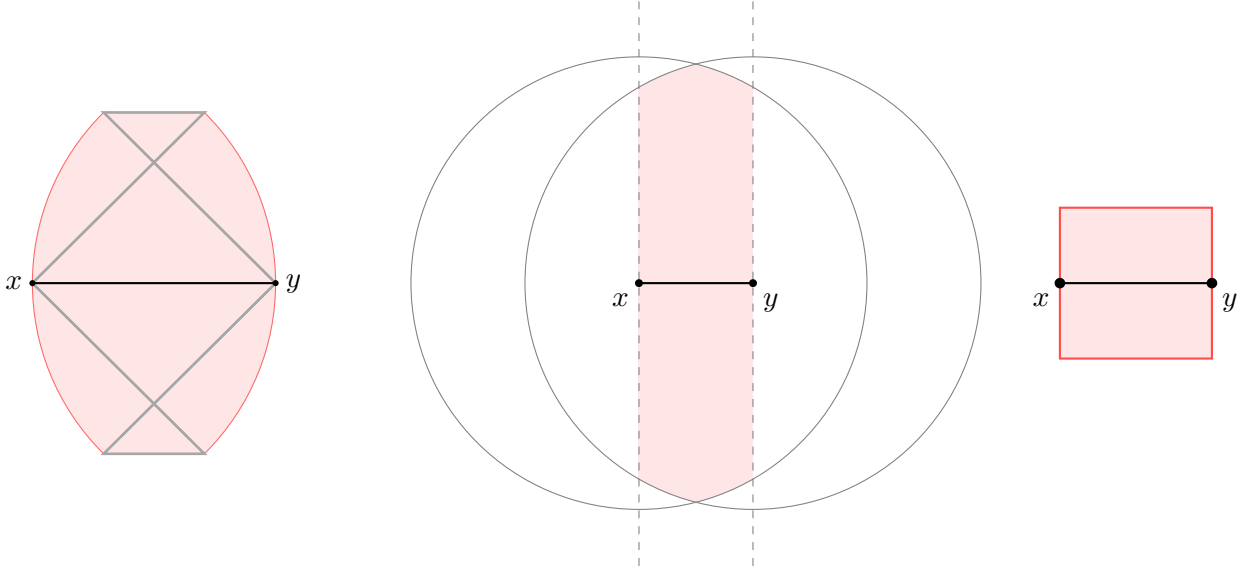
\begin{figure}[t]
\centering
\begin{minipage}{0.29\textwidth}
\centering
\begin{tikzpicture}[scale=0.8]
  \usetikzlibrary{calc,angles,quotes}

  \coordinate (x) at (0,0);
  \coordinate (y) at (4,0);
  \def\r{4}

  \coordinate (A) at ($(x)+(\r*0.70710678,\r*0.70710678)$); 
  \coordinate (B) at ($(x)+(\r*0.70710678,-\r*0.70710678)$); 
  \coordinate (C) at ($(y)+(-\r*0.70710678,\r*0.70710678)$); 
  \coordinate (D) at ($(y)+(-\r*0.70710678,-\r*0.70710678)$); 

  \fill[red!10, draw=red!60]
    (x) -- ($(x)+(45:\r)$) arc[start angle=45, end angle=-45, radius=\r] -- cycle;
  \fill[red!10, draw=red!60]
    (y) -- ($(y)+(135:\r)$) arc[start angle=135, end angle=225, radius=\r] -- cycle;

  \filldraw[fill=red!10, draw=gray!70, line width=1pt]
    (A) -- (C) -- (y) -- (D) -- (B) -- (x) -- cycle;
    

  \draw[gray!70, line width=0.8pt] (x) -- ($(x)+(45:\r*0.9)$);
  \draw[gray!70, line width=0.8pt] (x) -- ($(x)+(-45:\r*0.9)$);

  \draw[gray!70, line width=0.8pt] (y) -- ($(y)+(135:\r*0.9)$);
  \draw[gray!70, line width=0.8pt] (y) -- ($(y)+(225:\r*0.9)$);

    
\draw[thick] (x) -- (y);
\fill (x) circle (0.05) node[left] {$x$};
\fill (y) circle (0.05) node[right] {$y$};

\end{tikzpicture}
\end{minipage}%
\hfill
\begin{minipage}{0.49\textwidth}
\centering
\begin{tikzpicture}[scale=0.75]
  \coordinate (x) at (0,0);
  \coordinate (y) at (2,0);
  \def\r{4} 

  \begin{scope}
    \clip (x) circle (\r);
    \clip (y) circle (\r);
    \clip (-0.0,-5) rectangle (2.0,5); 
    \fill[red!10] (-5,-5) rectangle (7,5);
  \end{scope}

  \draw[gray] (x) circle (\r);
  \draw[gray] (y) circle (\r);
  \draw[gray, dashed] (0, -5) -- (0,5);
  \draw[gray, dashed] (2, -5) -- (2,5);

  \draw[thick] (x) -- (y);

  \fill (x) circle (0.07) node[below left] {$x$};
  \fill (y) circle (0.07) node[below right] {$y$};

\end{tikzpicture}
\end{minipage}%
\begin{minipage}{0.19\textwidth}
\centering
\begin{tikzpicture}[scale=1]
  \coordinate (x) at (0,0);
  \coordinate (y) at (2,0);
  \coordinate (m) at (1,0); 
  \def\r{2} 

  \draw[fill=red!10, draw=red!70, thick, rotate around={0:(m)}]
    (m) ++(-1,-1) rectangle ++(2,2);

  \draw[thick] (x) -- (y);

  \fill (x) circle (0.07) node[below left] {$x$};
  \fill (y) circle (0.07) node[below right] {$y$};
\end{tikzpicture}
\end{minipage}%
\caption{Two points $x,y$ and the regions $S(x,y)$ for the Pacman graph with $\beta=\pi/2$ (left), the truncated slab graph (middle), and the generalised Gabriel graph with $K=[-1/2,1/2]^2$ (right).}
\label{fig:Ex67}
\end{figure}

\begin{theorem}\label{th:PoisConv}
For any of the empty region graphs from Examples \ref{ex:Gabriel}--\ref{ex:ConvBody}, we have $\xi_t \overset{d}{\longrightarrow} \zeta$ as $t \to \infty$, where $\zeta$ is a Poisson process with intensity measure $\vol \otimes \mu \otimes\nu$. More precisely, for any set $W \in \mathcal{B}(\RR^d)$ with $\vol(W)<\infty$ and any $b \in \RR$, it holds that
    \[
    \mathrm{d}_{\mathrm{KR}} (\xi_t|_{W \times [b, \infty)\times\LL^d}, \zeta|_{W \times [b, \infty)\times\LL^d}) \le \begin{cases}
        c\log^{-(d-2)}(t), & d \ge 3,\\
        c\log^{-1/2}(t), & d=2,
    \end{cases}
    \]
for any $t \ge 2$. Here, $c\in(0,\infty)$ is a constant which only depends on $d$, $\vol(W)$, $b$, and the parameters of the particular example.
\end{theorem}

Note that the parameters of the particular examples that affect the constant in the previous theorem are $\beta$ for the {beta-skeleton} graph, $\theta$ for the Pacman graph, and $K$ for the generalised Gabriel graph.

Our second main result describes and quantifies the (to some extent) universal asymptotic distributional behaviour of the longest edge with midpoint in an observation window of each of the empty region graphs in Examples \ref{ex:Gabriel}--\ref{ex:ConvBody}. To this end, let
$$
L_{t,W} := \max\Big\{\|x-y\|:(x,y)\in\eta_{t,\neq}^2,(\eta_t-\delta_x-\delta_y)(S(x,y))=0,{x+y\over 2}\in W\Big\}
$$
be the length of the longest edge in such a graph with midpoint $\frac{x+y}{2}$ in a set $W \in \mathcal{B}(\RR^d)$ with $\vol(W)<\infty$. Here, we use the convention that $\max\varnothing=-\infty$. Furthermore, for two random variables $X$ and $Y$ we denote by
$$
\mathrm{d}_{\mathrm{Kol}}(X,Y) := \sup_{s\in\RR}|\PP(X\leq s) - \PP(Y\leq s)|
$$
the \textit{Kolmogorov distance} between $X$ and $Y$, and note that convergence in Kolmogorov distance implies convergence in distribution.

\begin{theorem}\label{th:gumb}
Assume the same set-up as in Theorem \ref{th:PoisConv}. Then, for any $t\geq 2$,
$$
\mathrm{d}_{\mathrm{Kol}}\big(\gamma tL_{t,W}^d-\log(t)-\log(\kappa_d/(2\gamma)),G_W\big) \leq \begin{cases}
        c \log^{-(d-2)}(t), & d \ge 3,\\
        c \log^{-1/2}(t), & d=2,
    \end{cases}
$$
where $G_W$ is a Gumbel distributed random variable with $\PP(G_W\leq s)=\exp(-\vol(W)e^{-s})$ for $s\in\RR$. Here, $c\in(0,\infty)$ is a constant which only depends on $d$, $\vol(W)$, and the parameters of the particular example.
\end{theorem}

The remainder of this paper is structured as follows. We begin by restating a Poisson process approximation bound from \cite{BSY} and deriving a new related quantitative Poisson limit theorem in Section \ref{sec:PoissonApprox}, which are essentially the tools to establish Theorems \ref{th:PoisConv} and \ref{th:gumb}, respectively. In Section \ref{sec:AbstracResults}, abstract versions of Theorems \ref{th:PoisConv} and \ref{th:gumb}, formulated under general assumptions on the sets $(S(x,y))_{x,y\in\mathbb{R}^d}$, are presented and shown by using the bounds for distributional approximations from Section \ref{sec:PoissonApprox}. Finally, in Section \ref{sec:Proofs}, we prove Theorems \ref{th:PoisConv} and \ref{th:gumb} in their concrete form by verifying, for Examples \ref{ex:Gabriel}--\ref{ex:ConvBody}, the assumptions introduced in Section \ref{sec:AbstracResults}.

\section{Poisson process and Poisson approximation}\label{sec:PoissonApprox}

For Poisson process approximations as in Theorem \ref{th:PoisConv} we rely on a result from \cite{BSY}, which we recall in this section. In order to provide a rate of convergence to the Gumbel distribution as in Theorem \ref{th:gumb}, we employ a related new bound for Poisson approximation, which we show in this section and which we think is of independent interest.  

Let $\mathbb{X}$ and $\mathbb{Y}$ be locally finite second countable Hausdorf spaces and, as above, let $\mathbf{N}_{\mathbb{X}}$ be the space of locally finite counting measures on $\mathbb{X}$. For fixed $k\in\mathbb{N}$ let $f:\mathbb{X}^k\times\mathbf{N}_{\mathbb{X}}\to\mathbb{Y}$ and $g:\mathbb{X}^k\times\mathbf{N}_{\mathbb{X}}\to\{0,1\}$ be measurable functions that are symmetric in the arguments from $\mathbb{X}^k$. Further, let $\eta$ be a Poisson process on $\mathbb{X}$ with a locally finite intensity measure $\mathbf{K}$. Denoting by $\eta^k_{\neq}$ the collection of $k$-tuples of distinct points of $\eta$, we are interested in the point process
\begin{align}\label{eq:DefXi}
\xi := \frac{1}{k!} \sum_{\mathbf{x}\in\eta^k_{\neq}} g(\mathbf{x},\eta) \delta_{f(\mathbf{x},\eta)},    
\end{align}
whose intensity measure is denoted by $\mathbf{L}$. Since $\xi$ depends on the underlying Poisson process {$\eta$}, we occasionally write $\xi(\,\cdot\,;\eta)$ whenever we would like to emphasise the role of $\eta$. Let $\mathcal{F}$ be the set of closed subsets of $\mathbb{Y}$ equipped with the Borel $\sigma$-field generated by the Fell topology, see \cite[Chapter 2]{SW}. Assume that there exists a measurable {symmetric} map $S: \mathbb{X}^k\to\mathcal{F}$ such that, {for all $\omega\in\mathbf{N}_{\mathbb{X}}$ and $\mathbf{x}\in\omega^k_{\neq}$,}
\begin{equation}\label{eqn:assumption_xS}
\mathbf{x}\in S_{\mathbf{x}}^k,
\end{equation}
\begin{equation}\label{eqn:assumption_g}
g(\mathbf{x},\omega) = g(\mathbf{x},\omega|_{S_{\mathbf{x}}}),
\end{equation}
and
\begin{equation}\label{eqn:assumption_f}
f(\mathbf{x},\omega) = f(\mathbf{x},\omega|_{S_{\mathbf{x}}}) \quad \text{if} \quad g(\mathbf{x},\omega)=1,
\end{equation}
where $S_{\mathbf{x}}:=S(\mathbf{x})$ and $\omega|_{S_{\mathbf{x}}}$ denotes the restriction of $\omega$ to $S_{\mathbf{x}}$. Here and in what follows, we write $\mathbf{x}=(x_1,\ldots,x_k)$ for a $k$-tuple of points of $\XX$ and use the short hand notation 
$$
\delta_{\mathbf{x}} {:=} \sum_{i=1}^k \delta_{x_i}.
$$

The following result is a simplified version of \cite[Theorem 5.1]{BSY} adapted to our situation.

\begin{theorem}\label{thm:BSY}
Let $\xi$ be the point process defined by \eqref{eq:DefXi}, and assume that $\mathbf{L}(\mathbb{Y})<\infty$, {\eqref{eqn:assumption_xS},} \eqref{eqn:assumption_g}, and \eqref{eqn:assumption_f} are satisfied. Let $\zeta$ be a Poisson process on $\mathbb{Y}$ with a finite intensity measure $\mathbf{M}$, and define
\begin{align*}
E_2 & := \frac{2}{(k!)^2} \int_{\mathbb{X}^k} \int_{\mathbb{X}^k} \mathbbm{1}\{S_{\mathbf{x}}\cap S_{\mathbf{z}}\neq\varnothing\} \mathbb{E}[ g(\mathbf{x},\eta+\delta_{\mathbf{x}})] \mathbb{E}[g(\mathbf{z},\eta+\delta_{\mathbf{z}}) ] \, \mathbf{K}^k(\dint \mathbf{z}) \, \mathbf{K}^k(\dint \mathbf{x}) \\
E_3 & := \frac{2}{(k!)^2} \int_{\mathbb{X}^k} \int_{\mathbb{X}^k} \mathbbm{1}\{S_{\mathbf{x}}\cap S_{\mathbf{z}}\neq\varnothing\} \mathbb{E}[ g(\mathbf{x},\eta+\delta_{\mathbf{x}}+\delta_{\mathbf{z}}) g(\mathbf{z},\eta+\delta_{\mathbf{x}}+\delta_{\mathbf{z}}) ] \, \mathbf{K}^k(\dint \mathbf{z}) \, \mathbf{K}^k(\dint \mathbf{x}) \\
E_4 & := \frac{2}{k!} \sum_{\varnothing \subsetneq I \subsetneq \{1,\hdots,k\}} \frac{1}{(k-|I|)!} \int_{\mathbb{X}^k} \int_{\mathbb{X}^{k-|I|}} \mathbb{E}[ g(\mathbf{x},\eta+\delta_{\mathbf{x}}+\delta_{\mathbf{z}}) g((\mathbf{x}_I,\mathbf{z}),\eta+\delta_{\mathbf{x}}+\delta_{\mathbf{z}}) ] \, \mathbf{K}^{k-|I|}(\dint \mathbf{z}) \, \mathbf{K}^k(\dint \mathbf{x})
\end{align*}
with $(\mathbf{x}_I,\mathbf{z}):=(x_{i_1},\hdots,x_{i_m},z_1,\hdots,z_{k-m})$ for $I=\{i_1,\hdots,i_m\}$.
Then,
$$
\mathrm{d}_{\mathrm{KR}}(\xi,\zeta) \leq \mathrm{d}_{\mathrm{TV}}(\mathbf{L},\mathbf{M}) + E_2 + E_3 + E_4.
$$
\end{theorem}

Analogously to \eqref{eqn:dTV_measures}, we define for non-negative integer-valued random variables $X$ and $Y$ the {\em total variation distance}
$$
\mathrm{d}_{\mathrm{TV}}(X,Y) = \sup_{A\subseteq \mathbb{N}_0} |\mathbb{P}(X\in A)-\mathbb{P}(Y\in A)|
$$
with $\mathbb{N}_0:=\mathbb{N}\cup\{0\}$. We will also use the following Poisson approximation result, whose proof is a combination of the arguments from the proof of Theorem \ref{thm:BSY} in \cite{BSY} and the Chen--Stein method for Poisson approximation, for which we refer to \cite{BarbourHolstJanson,Ross}.

\begin{theorem}\label{thm:Poisson-Approximation}
Let $\xi$ be the point process defined by \eqref{eq:DefXi}, and assume that $\mathbf{L}(\mathbb{Y})<\infty$, {\eqref{eqn:assumption_xS},} \eqref{eqn:assumption_g}, and \eqref{eqn:assumption_f} are satisfied. Further, let $Y$ be a Poisson random variable with parameter $a\in(0,\infty)$. Then,
$$
\mathrm{d}_{\mathrm{TV}}(\xi(\mathbb{Y}),Y) \leq  \min\{1,1/\sqrt{a}\} |\mathbb{E}[\xi(\mathbb{Y})] -a| + \frac{\min\{1,1/a\}}{2} (E_2 + E_3 + E_4)
$$
with $E_2$, $E_3$, and $E_4$ as in Theorem \ref{thm:BSY}.
\end{theorem}

\begin{remark}
{The upper bound in Theorem \ref{thm:BSY} also dominates} the total variation distance $\mathrm{d}_{\mathrm{TV}}(\xi(\mathbb{Y}),Y)$ considered in Theorem \ref{thm:Poisson-Approximation}. However, the bound in Theorem \ref{thm:Poisson-Approximation} is much stronger for large values of $a$ due to the {prefactors $\min\{1,1/\sqrt{a}\}$ and $\min\{1,1/a\}/2$.} Because of the counterintuitive behaviour to be smaller the larger the parameter $a$ of the Poisson distribution is, such factors are called magic factors.
\end{remark}

\begin{proof}[Proof of Theorem \ref{thm:Poisson-Approximation}]
Our proof relies on the Chen--Stein method for Poisson approximation. {For a function $h:\mathbb{N}_0\to\RR$ we} put
$$
\Delta h(k) := h(k+1) - h(k) \quad \text{for} \quad k\in\mathbb{N}_0 \qquad \text{and} \qquad \|h\|_{\infty} := \sup_{k\in\mathbb{N}_0} |h(k)|,
$$
and let $\mathcal{H}_a$ be the set of all functions $h:\mathbb{N}_0\to\RR$ for which
\begin{equation}\label{eq:ChenSteinSolutionBound}
\|h\|_\infty \leq \min\{1,1/\sqrt{a}\} \quad \text{and} \quad \| \Delta h \|_\infty \leq \min\{1,1/a\}.  
\end{equation}
Moreover, for $m,n\in\NN_0$ we have
\begin{equation}\label{eq:ChenSteinSolutionIncrement}
|h(m)-h(n)|\leq\|\Delta h\|_\infty|m-n|.
\end{equation}
From \cite[Theorem 4.5]{Ross} it follows that
\begin{equation}\label{eqn:BoundChenStein}
\mathrm{d}_{\mathrm{TV}}(\xi(\mathbb{Y}),Y) \leq \sup_{h\in\mathcal{H}_a} |\mathbb{E}[ a\, h(\xi(\mathbb{Y})+1) - \xi(\mathbb{Y}) h(\xi(\mathbb{Y})) ]|.
\end{equation}
In the following we consider the right-hand side for a fixed function $h\in\mathcal{H}_a$. First note that
\begin{equation}\label{eqn:bound_R0}
\big| \mathbb{E}[ a\, h(\xi(\mathbb{Y})+1) - \mathbb{E}[\xi(\mathbb{Y})] \mathbb{E}[h(\xi(\mathbb{Y})+1)] \big| \leq \|h\|_\infty |\mathbb{E}[\xi(\mathbb{Y})] -a| \leq \min\{1,1/\sqrt{a}\} |\mathbb{E}[\xi(\mathbb{Y})] -a|
\end{equation}
by \eqref{eq:ChenSteinSolutionBound}.
From the multivariate Mecke formula for Poisson processes \cite[Theorem 4.4]{LP} we obtain
$$
\mathbb{E}[\xi(\mathbb{Y}) h(\xi(\mathbb{Y}))] = \mathbb{E}\frac{1}{k!}\sum_{\mathbf{x}\in\eta^k_{\neq}} g(\mathbf{x},\eta) h(\xi(\mathbb{Y};\eta) = \frac{1}{k!}\int_{\mathbb{X}^k} \mathbb{E}[g(\mathbf{x},\eta+\delta_{\mathbf{x}}) h(\xi(\mathbb{Y};\eta+\delta_{\mathbf{x}})) ] \, \mathbf{K}^k(\dint\mathbf{x}),
$$
and
\begin{align*}
\mathbb{E}[\xi(\mathbb{Y})] \mathbb{E}[h(\xi(\mathbb{Y})+1)] & = \frac{1}{k!}\,\mathbb{E}\sum_{\mathbf{x}\in\eta^k_{\neq}} g(\mathbf{x},\eta) \;\mathbb{E}[h(\xi(\mathbb{Y})+1)] \\
& = \frac{1}{k!} \int_{\mathbb{X}^k} \mathbb{E}[g(\mathbf{x},\eta+\delta_{\mathbf{x}})] \mathbb{E}[h(\xi(\mathbb{Y})+1)] \, \mathbf{K}^k(\dint\mathbf{x}).
\end{align*}
Combining this with \eqref{eqn:bound_R0} shows that
\begin{align}
& |\mathbb{E}[ a\, h(\xi(\mathbb{Y})+1) - \xi(\mathbb{Y}) h(\xi(\mathbb{Y})) ]| \notag \\
& \leq |\mathbb{E}[ a\, h(\xi(\mathbb{Y})+1) - \mathbb{E}[\xi(\mathbb{Y})] \mathbb{E}[h(\xi(\mathbb{Y})+1)]| + |\mathbb{E}[\xi(\mathbb{Y})] \mathbb{E}[h(\xi(\mathbb{Y})+1)] - \mathbb{E}[\xi(\mathbb{Y}) h(\xi(\mathbb{Y})) ]| \notag \\
& \leq \min\{1,1/\sqrt{a}\} |\mathbb{E}[\xi(\mathbb{Y})] -a| \notag \\
& \quad + \bigg| \frac{1}{k!} \int_{\mathbb{X}^k} \mathbb{E}[g(\mathbf{x},\eta+\delta_{\mathbf{x}})] \mathbb{E}[h(\xi(\mathbb{Y})+1)] - \mathbb{E}[g(\mathbf{x},\eta+\delta_{\mathbf{x}}) h(\xi(\mathbb{Y};\eta+\delta_{\mathbf{x}})) ] \, \mathbf{K}^k(\dint\mathbf{x}) \bigg|. \label{eqn:expression_Stein}
\end{align}

Next, for $\mathbf{x}\in\XX^k$ we define
$$
Z_{\mathbf{x}} := \frac{1}{k!} \sum_{\mathbf{z}\in\eta^k_{\neq}: S_{\mathbf{x}}\cap S_{\mathbf{z}}=\varnothing} g(\mathbf{z},\eta).
$$
From \eqref{eq:ChenSteinSolutionIncrement} and the multivariate Mecke formula it follows that
\begin{align*}
\left| \mathbb{E}[h(\xi(\mathbb{Y})+1)] - \mathbb{E}[h(Z_{\mathbf{x}}+1)] \right| & \leq \|\Delta h\|_\infty \mathbb{E}[ |\xi(\mathbb{Y}) - Z_{\mathbf{x}}| ] \\
& \leq {\|\Delta h\|_\infty\over k!} \, \mathbb{E} \sum_{\mathbf{z}\in\eta^k_{\neq}: S_{\mathbf{x}}\cap S_{\mathbf{z}}\neq\varnothing} g(\mathbf{z},\eta) \\
& = {\|\Delta h\|_\infty \over k!} \int_{\mathbb{X}^k} \mathbbm{1}\{ S_{\mathbf{x}}\cap S_{\mathbf{z}}\neq\varnothing \} \mathbb{E}[g(\mathbf{z},\eta+\delta_{\mathbf{z}})]  \,\mathbf{K}^k(\dint\mathbf{z}).
\end{align*}
This implies that
\begin{align}\label{eqn:bound_R1}
\notag R_1 &:= 
\bigg|{1\over k!}\int_{\XX^k}\EE[g(\mathbf{x},\eta+\delta_{\mathbf{x}})]\,\big(\EE[ h(\xi(\mathbb{Y})+1)] - \mathbb{E}[h(Z_{\mathbf{x}}+1)]\big)\,\mathbf{K}^k(\dint\mathbf{x})\bigg|\\
\notag & \leq   \frac{\|\Delta h\|_\infty}{(k!)^2} \int_{\mathbb{X}^k} \int_{\mathbb{X}^k} \mathbbm{1}\{ S_{\mathbf{x}}\cap S_{\mathbf{z}}\neq\varnothing \} \mathbb{E}[g(\mathbf{x},\eta+\delta_{\mathbf{x}})]  \mathbb{E}[g(\mathbf{z},\eta+\delta_{\mathbf{z}})]  \, \mathbf{K}^k(\dint \mathbf{z}) \, \mathbf{K}^k(\dint \mathbf{x}) \\
& \leq \frac{\min\{1,1/a\}}{2} E_2
\end{align}
by \eqref{eq:ChenSteinSolutionBound} and the definition of $E_2$.
For $\mathbf{x}\in\mathbb{X}^k$, $g(\mathbf{x},\eta+\delta_{\mathbf{x}})$ and $Z_{\mathbf{x}}$ depend only on $\eta|_{S_{\mathbf{x}}}$ and $\eta|_{S_{\mathbf{x}}^c}$ by \eqref{eqn:assumption_g}, respectively. {By the independence property of Poisson processes,} $g(\mathbf{x},\eta+\delta_{\mathbf{x}})$ and $Z_{\mathbf{x}}$ are independent random variables.
It follows that $\mathbb{E}[g(\mathbf{x},\eta+\delta_{\mathbf{x}})] \mathbb{E}[h(Z_{\mathbf{x}}+1)]=\mathbb{E}[g(\mathbf{x},\eta+\delta_{\mathbf{x}}) h(Z_{\mathbf{x}}+1)]$ and hence
\begin{equation}\label{eqn:independence_integrals}
\frac{1}{k!} \int_{\mathbb{X}^k} \mathbb{E}[g(\mathbf{x},\eta+\delta_{\mathbf{x}})] \mathbb{E}[h(Z_{\mathbf{x}}+1)] \, \mathbf{K}^k(\dint\mathbf{x}) = \frac{1}{k!} \int_{\mathbb{X}^k} \mathbb{E}[g(\mathbf{x},\eta+\delta_{\mathbf{x}}) h(Z_{\mathbf{x}}+1)] \, \mathbf{K}^k(\dint\mathbf{x}).
\end{equation}
Moreover, we define
\begin{align*}
\widehat{Z}_{\mathbf{x}} := & \xi(\mathbb{Y};\eta+\delta_{\mathbf{x}}) - g(\mathbf{x},\eta+\delta_{\mathbf{x}}) - \frac{1}{k!} \sum_{\mathbf{z}\in\eta^k_{\neq}} g(\mathbf{z},\eta+\delta_{\mathbf{x}}) = \frac{1}{k!} \sum_{\mathbf{z}\in(\eta+\delta_{\mathbf{x}})^k_{\neq}\setminus(\eta_{\neq}^k\cup (\delta_{\mathbf{x}})^k_{\neq})} g(\mathbf{z},\eta+\delta_{\mathbf{x}}) \\
= & \sum_{\varnothing\subsetneq I \subsetneq \{1,\hdots,k\}} \frac{1}{(k-|I|)!} \sum_{\mathbf{z}\in\eta_{\neq}^{k-|I|}} g((\mathbf{x}_I,\mathbf{z}),\eta+\delta_{\mathbf{x}})
\end{align*}
for $\mathbf{x}\in\mathbb{X}^k$. Using this notation as well as the fact that, by {\eqref{eqn:assumption_xS} and} \eqref{eqn:assumption_g},
$$
\frac{1}{k!} \sum_{\mathbf{z}\in\eta^k_{\neq}: S_{\mathbf{x}}\cap S_{\mathbf{z}}=\varnothing} g(\mathbf{z},\eta) = \frac{1}{k!} \sum_{\mathbf{z}\in\eta^k_{\neq}: S_{\mathbf{x}}\cap S_{\mathbf{z}}=\varnothing} g(\mathbf{z},\eta+\delta_{\mathbf{x}}),
$$
we obtain
\begin{align*}
&|\xi(\mathbb{Y};\eta+\delta_{\mathbf{x}}) - (Z_{\mathbf{x}}+1)| \\
& = \bigg| \xi(\mathbb{Y};\eta+\delta_{\mathbf{x}}) - \frac{1}{k!} \sum_{\mathbf{z}\in\eta^k_{\neq}: S_{\mathbf{x}}\cap S_{\mathbf{z}}=\varnothing} g(\mathbf{z},\eta) -1 \bigg| \\
&=\bigg|\widehat{Z}_{\mathbf{x}} + g(\mathbf{x},\eta+\delta_{\mathbf{x}}) + \frac{1}{k!} \sum_{\mathbf{z}\in\eta^k_{\neq}} g(\mathbf{z},\eta+\delta_{\mathbf{x}}) - \frac{1}{k!} \sum_{\mathbf{z}\in\eta^k_{\neq}: S_{\mathbf{x}}\cap S_{\mathbf{z}}=\varnothing} g(\mathbf{z},\eta+\delta_{\mathbf{x}})-1\bigg|\\
& = \bigg| \widehat{Z}_{\mathbf{x}} + \frac{1}{k!} \sum_{\mathbf{z}\in\eta^k_{\neq}: S_{\mathbf{x}}\cap S_{\mathbf{z}}\neq\varnothing} g(\mathbf{z},\eta+\delta_{\mathbf{x}}) + g(\mathbf{x},\eta+\delta_{\mathbf{x}}) - 1 \bigg|\\
& \leq \widehat{Z}_{\mathbf{x}} + \frac{1}{k!} \sum_{\mathbf{z}\in\eta^k_{\neq}: S_{\mathbf{x}}\cap S_{\mathbf{z}}\neq\varnothing} g(\mathbf{z},\eta+\delta_{\mathbf{x}}) + |g(\mathbf{x},\eta+\delta_{\mathbf{x}}) - 1|.
\end{align*}
It follows, using \eqref{eq:ChenSteinSolutionIncrement}, that
\begin{align*}
R_2 & :=\bigg|\frac{1}{k!}\int_{\mathbb{X}^k} \mathbb{E}[g(\mathbf{x},\eta+\delta_{\mathbf{x}}) h(\xi(\mathbb{Y};\eta+\delta_{\mathbf{x}})) ] \, \mathbf{K}^k(\dint\mathbf{x}) - \frac{1}{k!} \int_{\mathbb{X}^k} \mathbb{E}[g(\mathbf{x},\eta+\delta_{\mathbf{x}}) h(Z_{\mathbf{x}}+1)] \, \mathbf{K}^k(\dint\mathbf{x})\bigg| \\
& \leq \frac{\|\Delta h\|_\infty}{k!}  \int_{\mathbb{X}^k} \mathbb{E}[g(\mathbf{x},\eta+\delta_{\mathbf{x}}) |\xi(\mathbb{Y};\eta+\delta_{\mathbf{x}}) - (Z_{\mathbf{x}}+1)| ] \, \mathbf{K}^k(\dint\mathbf{x}) \\
& \leq \frac{\|\Delta h\|_\infty}{k!}  \int_{\mathbb{X}^k} \mathbb{E}[g(\mathbf{x},\eta+\delta_{\mathbf{x}}) \widehat{Z}_{\mathbf{x}} ] \, \mathbf{K}^k(\dint\mathbf{x}) \\
& \quad + \frac{\|\Delta h\|_\infty}{(k!)^2}  \int_{\mathbb{X}^k} \mathbb{E}\bigg[g(\mathbf{x},\eta+\delta_{\mathbf{x}}) \sum_{\mathbf{z}\in\eta^k_{\neq}: S_{\mathbf{x}}\cap S_{\mathbf{z}}\neq\varnothing} g(\mathbf{z},\eta+\delta_{\mathbf{x}}) \bigg] \, \mathbf{K}^k(\dint\mathbf{x}).
\end{align*}
In the last step,  we used the fact that $g(\mathbf{x},\eta+\delta_{\mathbf{x}})|g(\mathbf{x},\eta+\delta_{\mathbf{x}})-1|=0$ as the function $g$ can only take the values $0$ and $1$. {Applying} the multivariate Mecke formula once again, together with the definition of $\widehat{Z}_{\mathbf{x}}$, we derive
\begin{equation}\label{eqn:bound_R2}
R_2 \leq \frac{\min\{1,1/a\}}{2} (E_3 + E_4)
\end{equation}
from \eqref{eq:ChenSteinSolutionBound} and the definitions of $E_3$ and $E_4$. From \eqref{eqn:expression_Stein} together with the definitions of $R_1$ and $R_2$ and \eqref{eqn:independence_integrals} we derive
$$
|\mathbb{E}[ a\, h(\xi(\mathbb{Y})+1) - \xi(\mathbb{Y}) h(\xi(\mathbb{Y})) ]| \leq \min\{1,1/\sqrt{a}\} |\mathbb{E}[\xi(\mathbb{Y})] -a| + R_1 + R_2.
$$
Now combining \eqref{eqn:BoundChenStein} with \eqref{eqn:bound_R1} and \eqref{eqn:bound_R2} completes the proof.
\end{proof}

\section{Longest edges in general empty region graphs}\label{sec:AbstracResults}

\subsection{Results}

The goal of this section is to derive general versions of Theorems \ref{th:PoisConv} and \ref{th:gumb}, which make some abstract assumptions on the sets $(S(x,y))_{x,y\in\mathbb{R}^d}$. In the next section, we will then verify these assumptions for Examples \ref{ex:Gabriel}--\ref{ex:ConvBody}, thus establishing Theorems \ref{th:PoisConv} and \ref{th:gumb}. Recall that we write $\LL^d$ for the space of lines in $\RR^d$ and $\mathrm{span}\,B$ for the linear hull of a set $B\subset\RR^d$, while $\nu$ denotes the rotation-invariant Haar probability measure on $\LL^d$ and $\mu$ is the measure on $\mathbb{R}$ such that $\mu([a,\infty))=e^{-a}$ for all $a\in\mathbb{R}$. We equip the closed subsets of $\mathbb{R}^d$ with the Borel $\sigma$-field generated by the Fell topology as in Section \ref{sec:PoissonApprox}.

\begin{theorem}\label{thm:PoisConvAbstract}
Let $\eta_t$ be a stationary Poisson process on $\RR^d$ with intensity $t\geq 2$. To each pair of distinct points $x,y\in\RR^d$ we associate a compact set $S(x,y)$ such that $(x,y)\mapsto S(x,y)$ is measurable and the following assumptions for pairs of distinct points $(x,y),(u,v)\in(\RR^d)^2$ are satisfied:
\begin{itemize}
    \item[(i)] $\vol(S(x,y))=\gamma\|x-y\|^d$ for some $\gamma\in(0,\infty)$,
    \item[(ii)] $S(x,y)\subseteq B^d({x+y\over 2},\alpha\|x-y\|)$ for some {$\alpha\in(\frac{1}{2},\infty)$,}
    \item[(iii)] $\vol(S(x,y)\setminus S(u,v)) \geq \beta\|{x+y\over 2}-{u+v\over 2}\|\|x-y\|^{d-1}$ for some $\beta\in(0,\infty)$ if $\|x-y\|\geq\|u-v\|$ and $S(x,y)\cap S(u,v)\neq\varnothing$,
    \item[(iv)] $S(x,y)=S(y,x)$.
\end{itemize}
Define the point process
$$
\xi_t := {1\over 2}\sum_{(x,y)\in\eta_{t,\neq}^2}\mathbbm{1}\{(\eta_t-\delta_x-\delta_y)(S(x,y))=0\}\delta_{\big({x+y\over 2},\,\gamma t\|x-y\|^d-\log(t)-\log(\kappa_d/(2\gamma)),\,\mathrm{span}\{x-y\}\big)}
$$
on $\RR^d\times\RR\times\LL^d$, denote by $\zeta$ a Poisson process on $\RR^d\times\RR\times\LL^d$ with intensity measure $\vol\otimes\mu\otimes\nu$, and let $W\in\mathcal{B}(\RR^d)$ with $\vol(W)<\infty$ and $b\in\RR$. Then, for any $t\geq 2$,
\begin{equation}\label{eqn:bound_dKR}
\mathrm{d}_{\mathrm{KR}}(\xi_t|_{W\times[b,\infty)\times\LL^d},\zeta|_{W\times[b,\infty)\times\LL^d}) \leq \begin{cases}
        c\log^{-(d-2)}(t), & d \ge 3,\\
        c\log^{-1/2}(t), & d=2,
    \end{cases}
\end{equation}
where $c\in(0,\infty)$ is a constant only depending on $d$, $\vol(W)$, $b$, $\alpha$, $\beta$, and $\gamma$. {Moreover, one has $\xi_t\overset{d}{\longrightarrow} \zeta$ as $t\to\infty$.}
\end{theorem}

Using the Poisson approximation result Theorem \ref{thm:Poisson-Approximation}, we can also derive the following quantitative bound for the Gumbel approximation of the length of the longest edge with midpoint in an observation window for an empty region graph based on sets $S(x,y)_{x,y\in\mathbb{R}^d}$ satisfying the assumptions of the previous theorem.

\begin{theorem}\label{thm:GumbelAbstract}
Assume the same set-up as in Theorem \ref{thm:PoisConvAbstract} and denote by $L_{t,W}$ the length of the longest edge with midpoint in $W$, i.e., 
$$
L_{t,W} := \max\Big\{\|x-y\|:(x,y)\in\eta_{t,\neq}^2,(\eta_t-\delta_x-\delta_y)(S(x,y))=0,{x+y\over 2}\in W\Big\}. 
$$
Then, for any $t\geq 2$,
$$
\mathrm{d}_{\mathrm{Kol}}\big(\gamma tL_{t,W}^d-\log(t)-\log(\kappa_d/(2\gamma)),G_W\big) \leq \begin{cases}
        c\log^{-(d-2)}(t), & d \ge 3,\\
        c\log^{-1/2}(t), & d=2,
    \end{cases}
$$
where $c$ is a constant only depending on $d$, $\vol(W)$, $\alpha$, $\beta$, and $\gamma$, and $G_W$ is a Gumbel distributed random variable with distribution function $\PP(G_W\leq s)=\exp(-\vol(W)e^{-s})$ for $s\in\RR$.
\end{theorem}

The proofs of Theorem \ref{thm:PoisConvAbstract} and Theorem \ref{thm:GumbelAbstract} will be given in Section \ref{subsec:AbstractProofs} below after having developed some preparatory estimates, which are the content of the next section.

\subsection{Preparatory estimates}\label{subsec:AbstractPrep}

The core of the proof of Theorem \ref{thm:PoisConvAbstract} is an application of Theorem \ref{thm:BSY}, while Theorem \ref{thm:GumbelAbstract} will be obtained using Theorem \ref{thm:Poisson-Approximation}. In particular, we need to find suitable estimates for the quantities {$E_2$, $E_3$, and $E_4$} appearing therein. To start, we determine the intensity measure of the point process $\xi_t$. Throughout the proofs, we use the abbreviation $\tau:=\log(\kappa_d/(2\gamma))$.

\begin{lemma}\label{lem:intmeas}
Let the assumptions of Theorem \ref{thm:PoisConvAbstract} be satisfied. For all $t\geq 2$ the intensity measure of $\xi_t$ is given by
$$
\mathbf{M}_t(U):=\mathbb{E}[ \xi_t(U) ] = (\vol\otimes \mu|_{[-\log(t)-\tau,\infty)}\otimes\nu)(U)
$$
for $U\in\mathcal{B}(\RR^d\times\RR\times\LL^d)$. 
\end{lemma}
\begin{proof}
Let $A \in\mathcal{B}(\RR^d)$ with $\vol(A)<\infty$, let $D\subseteq\LL^d$ be measurable, and let $a\in\RR$. By the multivariate Mecke formula, the void probability of a stationary Poisson process, and assumption (i) of Theorem \ref{thm:PoisConvAbstract} we have
\begin{align*}
    &\mathbb{E}[\xi_t(A \times [a, \infty)\times D)]\\
    &= \frac{t^2}{2} \int_{\RR^d}\int_{\RR^d} \mathbb{P}(\eta_t(S(x,y)) = 0) \mathbbm{1}\Big\{\frac{x+y}{2} \in A, \gamma t\lVert x - y \rVert^d - \log(t) - \tau \ge a,\mathrm{span}\{x-y\}\in D\Big\} \,\dint x \, \dint y\\
		    &= \frac{t^2}{2} \int_{\RR^d}\int_{\RR^d} e^{-t\vol(S(x,y))} \mathbbm{1}\Big\{\frac{x+y}{2} \in A, \gamma t\lVert x - y \rVert^d - \log(t) - \tau \ge a,\mathrm{span}\{x-y\}\in D\Big\} \,\dint x \, \dint y\\
&= \frac{t^2}{2} \int_{\RR^d}\int_{\RR^d} e^{-\gamma t\|x-y\|^d} \mathbbm{1}\Big\{\frac{x+y}{2} \in A, \gamma t\lVert x - y \rVert^d - \log(t) - \tau \ge a,\mathrm{span}\{x-y\}\in D\Big\} \,\dint x \, \dint y.
\end{align*}
Now the substitution $z=(x+y)/2$ and $u=x-y$ leads to 
\begin{align*}
\mathbb{E}[\xi_t(A \times [a, \infty)\times D)] & = \frac{t^2}{2} \int_{\RR^d}\int_{\RR^d} e^{-\gamma t\|u\|^d} \mathbbm{1}\Big\{z \in A, \gamma t\lVert u \rVert^d - \log(t) - \tau \ge a,\mathrm{span}\{u\}\in D\Big\} \,\dint z \, \dint u \\
& = \frac{t^2 \vol(A)}{2} \int_{\RR^d} e^{-\gamma t\lVert u \rVert^d} \mathbbm{1}\bigg\{ \lVert u \rVert^d  \ge \max\bigg\{\frac{\log(t)+\tau+a}{\gamma t},0\bigg\},\mathrm{span}\{u\}\in D\bigg\}\,\dint u.
\end{align*}
Using spherical coordinates and the abbreviation
$$
a_t:= \max\Big\{\frac{\log(t)+\tau+a}{\gamma t},0\Big\}^{1/d},
$$
we obtain 
$$
   \mathbb{E}[\xi_t(A \times [a, \infty)\times D)] 
    = \frac{t^2 \vol(A)}{2} d \kappa_d\nu(D) \int_{a_t}^\infty e^{-t \gamma r^d} r^{d-1} \, \dint r = \frac{t \vol(A)\nu(D) \kappa_d}{2\gamma} e^{-t\gamma  a_t^d}.$$
If $a_t>0$, we have
$$
\frac{\kappa_d t}{2\gamma}e^{-t \gamma a_t^d} = \frac{\kappa_d t}{2\gamma}e^{-(\log(t)+\tau+a)} = \frac{\kappa_d}{2\gamma} \frac{2\gamma}{\kappa_d} e^{-a} = e^{-a} 
$$
so that
$$
\mathbb{E}[\xi_t(A \times [a, \infty)\times D)] = \begin{cases} \vol(A)\nu(D) e^{-a}, &  a> -\log(t) - \tau, \\ \frac{\kappa_d\vol(A)\nu(D)t}{2\gamma}, & a\le - \log(t) - \tau.  \end{cases}
$$
This can be rewritten as
$$
\mathbb{E}[\xi_t(A \times [a, \infty)\times D)] = \vol(A)\nu(D) \mu|_{[-\log(t)-\tau,\infty)}([a,\infty)).
$$
As the measure $\mathbf{M}_t$ is determined by the sets of the form $A\times [a,\infty)\times D$, this proves the formula for $\mathbf{M}_t$.
\end{proof}

Defining
$$
b_t:= \max\Big\{\frac{\log(t)+\tau+b}{\gamma t},0\Big\}^{1/d}
$$
and then
\[
g_{b,t}(x,y,\eta_t) := \mathbbm{1}\Big\{(\eta_t - \delta_x - \delta_y)(S(x,y)) = 0, \frac{x+y}{2} \in W,\lVert x - y\rVert \ge b_t\Big\},
\]
we have that
\[
\xi_t|_{W \times [b,\infty)\times \LL^d} = \frac{1}{2} \sum_{(x,y) \in \eta^2_{t,\ne}} g_{b,t}(x,y,\eta_t) \delta_{\big(\frac{x+y}{2}, \gamma t\lVert x - y\rVert^d - \log(t) - \tau,\mathrm{span}\{x-y\}\big)},
\]
which is exactly the form required in Theorem \ref{thm:BSY}. In what follows, we derive bounds on the quantities $E_2$, $E_3$, and $E_4$ appearing there.

\begin{lemma}\label{lem:estimates}
    Suppose that the assumptions of Theorem \ref{thm:PoisConvAbstract} are satisfied and that $\log(t)+\tau+b>0$. Then, for the point process $\xi_t|_{W \times [b,\infty)\times\LL^d}$, the quantities $E_2$, $E_3$, and $E_4$ in Theorem \ref{thm:BSY} {with $S_{(x,y)}:=S(x,y)\cup\{x,y\}$ for $x,y\in\mathbb{R}^d$} are bounded by
    \[
    E_2 \le C_2 \frac{1+ b+ \log(t) + \tau}{e^{2b}t}, \quad E_3 \le \frac{C_3}{e^b(b+\log(t)+\tau)^{\max\{d-2, 1/2\}}},\quad
    E_4 \le \frac{C_{4,1}}{e^b(b+\log(t)+\tau)^{d-1}}+\frac{C_{4,2}}{e^{2b}t}.
    \]
    Here, $C_2$, $C_3$, $C_{4,1}$, and $C_{4,2}$ are constants in $(0,\infty)$ depending only on $d$ and $W$ as well as the constants $\alpha,\beta,\gamma$ in Theorem \ref{thm:PoisConvAbstract}. Possible choices are
    \[
    C_2 := \frac{2^{d+1} \kappa_d \alpha^d}{\gamma}\vol(W),\quad C_3 := \begin{cases}
        \frac{4 d \kappa_d^2 M_{d-1}\gamma^{d-2}}{\beta^d(d-2)}\vol(W), & d \ge 3,\\
        {\frac{16\sqrt{2} \pi^{5/2}}{\beta^2\sqrt{\gamma}}\vol(W),} & d = 2,
    \end{cases}\quad
    C_{4,1} := \frac{2^{d+4} d \kappa_d M_{d-1}\gamma^{d-1}}{\beta^d} \vol(W), 
    \]
and $C_{4,2}:=16\vol(W)$, where $M_{d-1} := (2(d-1)/e)^{d-1}$.
\end{lemma}
\begin{proof}
We deal with the terms $E_2$, $E_3$, and $E_4$ separately.

\paragraph{Bound for $E_2$.} 
By assumption (i) of Theorem \ref{thm:PoisConvAbstract} we have
\begin{align*}
    \mathbb{E}[g_{b,t}(x,y,\eta_t+\delta_x+\delta_y)] &=\mathbbm{1}\Big\{{x+y\over 2}\in W,\lVert x - y\rVert \ge b_t\Big\}\,\PP(\eta_t(S(x,y))=0)\\
    &=\mathbbm{1}\Big\{{x+y\over 2}\in W,\lVert x - y\rVert \ge b_t\Big\}\,e^{-\gamma t \lVert x - y \rVert^d},
\end{align*}
and similarly for $\mathbb{E}[g_{b,t}(u,v,\eta_t+\delta_u+\delta_v)]$. For $E_2$ it follows that
\begin{align*}
    E_2 &= \frac{t^4}{2} \int_{(\RR^d)^2} \int_{(\RR^d)^2} \mathbbm{1}\{ {(S(x,y)\cup\{x,y\}) \cap (S(u,v)\cup\{u,v\})} \ne \varnothing \} \,\mathbb{E}[g_{b,t}(x,y,\eta_t+\delta_x+\delta_y)]\\
    &\hspace{3cm}\times \mathbb{E}[g_{b,t}(u,v,\eta_t+\delta_u+\delta_v)] \, \dint(x,y) \, \dint(u,v)\\
    &= \frac{t^4}{2} \int_{(\RR^d)^4} \mathbbm{1}\Big\{ {(S(x,y)\cup\{x,y\}) \cap (S(u,v)\cup\{u,v\})} \ne \varnothing, \frac{x+y}{2} \in W, \frac{u+v}{2} \in W, \\
    &\hspace{2.5cm} \lVert x - y\rVert \ge b_t,\lVert u - v\rVert \ge b_t\Big\} \, e^{-\gamma t \lVert x - y \rVert^d} e^{-\gamma t \lVert u - v \rVert^d} \,\dint(x,y,u,v).
\end{align*}
By assumption (ii) of Theorem \ref{thm:PoisConvAbstract} and since $\alpha > 1/2$, ${(S(x,y)\cup\{x,y\}) \cap (S(u,v)\cup\{u,v\})} \ne \varnothing$ implies that $B^d(\frac{x+y}{2}, \alpha \lVert x - y\rVert) \cap B^d(\frac{u+v}{2}, \alpha \lVert u - v\rVert) \ne \varnothing$. Therefore, the last integral may be estimated from above by
\begin{align*}
    & \frac{t^4}{2} \int_{(\RR^d)^4} \mathbbm{1}\Big\{ B^d\Big(\frac{x+y}{2}, \alpha \lVert x - y\rVert\Big) \cap B^d\Big(\frac{u+v}{2}, \alpha \lVert u - v\rVert\Big) \ne \varnothing, \frac{x+y}{2} \in W, \frac{u+v}{2} \in W,\\
    &\hspace{2.0cm} \lVert x - y\rVert \ge b_t, \lVert u - v\rVert \ge b_t \Big\} \,e^{- \gamma t \lVert x - y \rVert^d} e^{-\gamma t \lVert u - v \rVert^d} \,\dint(x,y,u,v)\\
    &= \frac{t^4}{2} \int_{(\RR^d)^4} \mathbbm{1}\{ B^d(z, \alpha \lVert z'\rVert) \cap B^d(w, \alpha \lVert w'\rVert) \ne \varnothing, z \in W, w \in W,  \lVert z'\rVert \ge b_t,\lVert w'\rVert \ge b_t \}\\
    &\hspace{1.5cm} \times e^{-\gamma t \lVert z' \rVert^d} e^{-\gamma t \lVert w' \rVert^d} \,\dint(z,z',w,w'),
\end{align*}
where we applied the substitution $z=(x+y)/2$, $w=(u+v)/2$, $z'=x-y$, and $w'=u-v$. {From} spherical coordinates, it follows that
\begin{align*}
    E_2 &\leq \frac{t^4}{2} d^2 \kappa_d^2 \int_W \int_W \int_0^\infty \int_0^\infty \mathbbm{1}\{ B^d(z, \alpha r) \cap B^d(w, \alpha s) \ne \varnothing, r\geq b_t,s\geq b_t \}\\
    &\hspace{3cm} \times e^{-\gamma t r^d} e^{-\gamma t s^d} (rs)^{d-1} \,\dint r \, \dint s \, \dint z \, \dint w \allowdisplaybreaks\\
    & \leq \frac{t^4}{2} d^2 \kappa_d^2 \int_W \int_{b_t}^\infty \int_{b_t}^\infty e^{-\gamma t (r^d +s^d)} (rs)^{d-1} \Big(\int_{\RR^d} \mathbbm{1} \{w \in B^d(z, \alpha (r+s)) \} \, \dint w\Big)\,\dint r\, \dint s \, \dint z\allowdisplaybreaks\\
    &= \frac{t^4}{2} d^2 \kappa_d^2 \vol(W) \int_{b_t}^\infty \int_{b_t}^\infty e^{-\gamma t (r^d +s^d)} (rs)^{d-1} \kappa_d\alpha^d(r+s)^d \,\dint r \, \dint s\allowdisplaybreaks\\
    &\le \frac{t^4}{2} d^2 \kappa_d^3 \vol(W) \int_{b_t}^\infty \int_{b_t}^\infty e^{-\gamma t (r^d +s^d)} (rs)^{d-1} 2^{d-1}\alpha^d(r^d+s^d) \,\dint r \, \dint s\allowdisplaybreaks\\
    & {=} 2^{d-1} t^4 d^2 \kappa_d^3 \alpha^d \vol(W) \int_{b_t}^\infty \int_{b_t}^\infty e^{-\gamma t (r^d +s^d)} r^{2d-1} s^{d-1} \,\dint r \, \dint s,
\end{align*}
where we used symmetry in the last step. Simple calculations with $\tilde{b}_t := \gamma t b_t^d = \log(t)+b+\tau$, where the equality holds due to $\log(t)+b+\tau>0$, now yield
\begin{align*}
\int_{b_t}^\infty e^{-\gamma t r^d} r^{2d-1} \, \dint r & =  \frac{1}{d(\gamma t)^2} \int_{\tilde{b}_t}^\infty ue^{-u} \, \dint u =  \frac{(1+\tilde{b}_t) e^{-\tilde{b}_t}}{d(\gamma t)^2} = \frac{(1+  b+ \log(t) + \tau) e^{-\log(t)-b-\tau}}{d(\gamma t)^2} \\
&  = \frac{2(1+  b+ \log(t) + \tau) e^{-b}}{d\kappa_d \gamma t^3}
\end{align*}
and
\begin{equation}\label{eqn:integral_b_t}
 \int_{b_t}^\infty e^{-\gamma t s^d} s^{d-1}  \, \dint s = \frac{1}{d\gamma t}  \int_{\tilde{b}_t}^\infty e^{-v} \, \dint v = \frac{e^{-\tilde{b}_t}}{d\gamma t} = \frac{2e^{-b}}{d\kappa_d t^2}.
\end{equation}
Altogether, this leads to the estimate
\[
E_2 \le \frac{2^{d+1} \kappa_d \alpha^d\vol(W)}{\gamma } \frac{e^{-2b}(1+ b+ \log(t) + \tau)}{t} = C_2 \frac{1+ b+ \log(t) + \tau}{e^{2b}t}.
\]

\paragraph{Bound for $E_3$.} 
Turning to $E_3$, we have
    \begin{align}\allowdisplaybreaks
    E_3 &= \frac{t^4}{2} \int_{(\RR^d)^2} \int_{(\RR^d)^2} \mathbbm{1}\{ {(S(x,y)\cup\{x,y\}) \cap (S(u,v)\cup\{u,v\})} \ne \varnothing \} \notag\\
    &\hspace{1.5cm} \times\mathbb{E}[g_{b,t}(x,y,\eta_t+\delta_x+\delta_y+\delta_u+\delta_v) g_{b,t}(u,v,\eta_t+\delta_x+\delta_y+\delta_u+\delta_v)] \,\dint(x,y) \, \dint(u,v) \notag \\
    &= \frac{t^4}{2} \int_{(\RR^d)^4} \mathbbm{1}\Big\{{(S(x,y)\cup\{x,y\}) \cap (S(u,v)\cup\{u,v\})} \ne \varnothing, u,v \notin S(x,y), x,y \notin S(u,v), \notag \\
    &\hspace{2.5cm} \frac{x+y}{2} \in W, \frac{u+v}{2} \in W,\lVert x-y \rVert \ge b_t, \lVert u-v \rVert\ge b_t\Big\}\, e^{-t\vol(S(x,y) \cup S(u,v))} \,\dint(x,y,u,v) \notag \\
    &= t^4 \int_{(\RR^d)^4} \mathbbm{1}\Big\{S(x,y) \cap S(u,v) \ne \varnothing, u,v \notin S(x,y), x,y \notin S(u,v), \frac{x+y}{2} \in W, \frac{u+v}{2} \in W, \notag \\
    &\hspace{2.5cm}\lVert x-y \rVert \ge \lVert u-v \rVert\ge b_t\Big\}\, e^{-t\vol(S(x,y) \cup S(u,v))} \,\dint(x,y,u,v), \label{eqn:E_3_intermediate}
		\end{align}
where we used symmetry in the last step. Next we write
\begin{align*}
\vol(S(x,y) \cup S(u,v)) = \vol(S(u,v)) + \vol(S(x,y) \setminus S(u,v)) .
\end{align*}
By assumptions (i) and (iii) of Theorem \ref{thm:PoisConvAbstract} we thus have
\begin{equation}\label{eqn:lower_bound_vol}
    \vol(S(x,y) \cup S(u,v) \geq \gamma\|u-v\|^d + \beta\bigg\|\frac{x+y}{2}-\frac{u+v}{2}\bigg\| \|x-y\|^{d-1}
\end{equation}
{if} $\|x-y\|\geq\|u-v\|$ and $S(x,y) \cap S(u,v) \ne \varnothing$.

In the following, we consider the cases $d\ge3$ and $d=2$ separately. For $d\ge3$ it follows from \eqref{eqn:E_3_intermediate} and \eqref{eqn:lower_bound_vol} that
\begin{align*}\allowdisplaybreaks
E_3 & \leq t^4 \int_{(\RR^d)^4} \mathbbm{1}\Big\{\frac{u+v}{2} \in W,\lVert x-y \rVert \ge \lVert u-v \rVert\ge b_t\Big\}\\
    &\hspace{2.5cm} \times \exp\bigg(-t\gamma\|u-v\|^d - t \beta\bigg\|\frac{x+y}{2}-\frac{u+v}{2}\bigg\| \|x-y\|^{d-1}\bigg) \,\dint(x,y,u,v) \\
		& = t^4 \int_{(\RR^d)^4} \mathbbm{1}\Big\{w \in W,\lVert z' \rVert \ge \lVert w' \rVert\ge b_t\Big\} e^{-t\gamma\|w'\|^d - t \beta\|z\| \|z'\|^{d-1}} \,\dint(z,z',w,w'),
\end{align*}
where we have set $z= (x+y)/2-(u+v)/2$, $w=(u+v)/2$, $z' = x-y$, {and} $w' = u-v$ in the second step. This inequality can be further rewritten as
\begin{align*}
E_3 & \leq t^4 \vol(W) \int_{(\RR^d)^3} \mathbbm{1}\Big\{\lVert z' \rVert \ge \lVert w' \rVert\ge b_t\Big\} e^{-t\gamma\|w'\|^d - t \beta\|z\| \|z'\|^{d-1}} \,\dint(z,z',w') \\
&= t^4\vol(W) (d\kappa_d)^3 \int_{[0, \infty)^3} \mathbbm{1}\{r \ge s\geq b_t\} e^{-t\gamma s^d - t \beta v r^{d-1}} (rsv)^{d-1} \,\dint(r,s,v)\\
    &= t^4\vol(W) (d\kappa_d)^3 \int_{[0, \infty)^3} \mathbbm{1}\{r \ge s\ge b_t\} e^{-t\gamma s^d} (rs)^{d-1}\,e^{-t\beta r^{d-1}v} (t\beta r^{d-1} v)^{d-1} \frac{\dint(r,s,v)}{(t\beta r^{d-1})^{d-1}} .
\end{align*}
Writing $M_{d-1} := \max_{p \ge 0} p^{d-1}e^{-p/2} = (2(d-1)/e)^{d-1}$, we have
\[
e^{-t\beta r^{d-1}v} (t\beta r^{d-1} v)^{d-1} \le M_{d-1}e^{-t\beta r^{d-1}v/2}
\]
so that we may proceed by bounding the last integral from above by
\begin{align*}
&  t^4\vol(W) (d\kappa_d)^3 M_{d-1}\int_{[0, \infty)^3} \mathbbm{1}\{ s \ge b_t, r \ge s\} e^{-t\gamma s^d} (rs)^{d-1} e^{-t\beta r^{d-1}v/2}\,\frac{\dint(r,s,v)}{(t\beta r^{d-1})^{d-1}}  \\
& = \frac{2 M_{d-1}\vol(W)(d\kappa_d)^3t^3}{\beta} \int_{[0, \infty)^2} \mathbbm{1}\{s \ge b_t, r \ge s\} e^{-t\gamma s^d} s^{d-1} \frac{1}{(t\beta r^{d-1})^{d-1}} \, \dint(r,s)\allowdisplaybreaks\\
    &= \frac{2M_{d-1}\vol(W)(d\kappa_d)^3t^{4-d}}{\beta^d ((d-1)^2-1)} \int_{b_t}^\infty e^{-t\gamma s^d} s^{d-1} \frac{\dint s}{s^{(d-1)(d-1)-1}} \\
    &\le \frac{2M_{d-1}\vol(W)(d\kappa_d)^3t^2}{\beta^d d(d-2)t^{d-2}b_t^{d(d-2)}} \int_{b_t}^\infty e^{-t\gamma s^d}s^{d-1} \,\dint s.
\end{align*}
Observe that the integrals on the left-hand side are divergent for $d=2$. Note that as already seen in \eqref{eqn:integral_b_t}, the last integral equals $2e^{-b}/(d\kappa_dt^2)$. Altogether, we obtain the bound
\[
E_3 \le \frac{4M_{d-1}(d\kappa_d)^2\vol(W)e^{-b}}{\beta^d d(d-2)((b+\log(t)+\tau)/\gamma)^{d-2}} = \frac{C_3}{e^b(b+\log(t)+\tau)^{d-2}}
\]
for $d\ge3$.

To address the case of $d=2$, we combine the trivial bound 
$$
\vol(S(x,y) \cup S(u,v))\ge \vol(S(x,y))=\gamma \|x-y\|^2,
$$
which follows from assumption (i) of Theorem \ref{thm:PoisConvAbstract}, with \eqref{eqn:lower_bound_vol} to obtain
\begin{align*}
\vol(S(x,y)\cup S(u,v)) & \geq \frac{\gamma\|x-y\|^2}{2} + \frac{1}{2} \Big(\beta \Big\lVert \frac{x+y}{2} - \frac{u+v}{2} \Big\rVert \lVert x-y \rVert + \gamma \lVert u-v \rVert^2 \Big) \\
& = \gamma\|u-v\|^2+ \frac{\beta}{2} \Big\lVert \frac{x+y}{2} - \frac{u+v}{2} \Big\rVert \lVert x-y \rVert + \frac{\gamma}{2} \big(\lVert x-y \rVert^2 - {\|u-v\|^2}  \big)
\end{align*}
if $\|x-y\|\ge\|u-v\|$ and $S(x,y)\cap S(u,v)\neq\varnothing$. Together with \eqref{eqn:E_3_intermediate} and the same arguments as for $d\ge3$, we derive
\begin{align*}
  E_3 & \le t^4\vol(W) (2\kappa_2)^3 \int_{[0, \infty)^3} \mathbbm{1}\{ r \ge s\ge b_t\} e^{-t\gamma s^2 - t \beta vr/2 - t\gamma (r^2-s^2)/2} rsv \,\dint(r,s,v)\\
		& = \frac{4t^2\vol(W) (2\kappa_2)^3}{\beta^2} \int_{[0, \infty)^2} \mathbbm{1}\{ r \ge s\ge b_t\} e^{-t\gamma s^2 - t\gamma (r^2-s^2)/2} \frac{s}r{} \,\dint(r,s)\\
		& \leq \frac{32t^2\vol(W) \kappa_2^3}{\beta^2 b_t} \int_{b_t}^\infty e^{-t\gamma s^2} {s} \int_s^\infty e^{- t\gamma (r^2-s^2)/2} \, \dint r  \, \dint s.
\end{align*}
Now, we note that by
{$$
\int_s^\infty e^{- t\gamma (r^2-s^2)/2} \, \dint r = \int_0^\infty e^{- t\gamma ((u+s)^2-s^2)/2} \, \dint u \leq \int_0^\infty e^{- t\gamma u^2/2} \, \dint u = \frac{\sqrt{2\pi}}{2} \frac{1}{\sqrt{t\gamma}}
$$}
for any $s \ge 0$ and by \eqref{eqn:integral_b_t} the integral is bounded by {$\sqrt{2\pi}e^{-b}/(2\kappa_2\sqrt{\gamma}t^{5/2})$}. This yields
\[
E_3 \le \frac{16 \sqrt{2\pi} \kappa_2^2  \vol(W) e^{-b}}{\beta^2 \sqrt{\gamma}\sqrt{t} b_t} =\frac{16 \sqrt{2\pi} \kappa_2^2 \vol(W) e^{-b}}{\beta^2\sqrt{\gamma}(b+\log(t)+\tau)^{1/2}} = \frac{C_3}{e^b(b+\log(t)+\tau)^{1/2}}
\]
for $d=2$.

\paragraph{Bound for $E_4$.} Finally, we bound the quantity $E_4$. We have
\begin{align*}
    E_4 &= 2t^3 \int_{(\RR^d)^3} \mathbb{E} [g_b(x,y,\eta_t+\delta_x+\delta_y+\delta_z)g_b(x,z,\eta_t+\delta_x+\delta_y+\delta_z)] \,\dint(x,y,z)\\
    &= 2 t^3 \int_{(\RR^d)^3} \mathbbm{1} \Big\{z \notin S(x,y),y \notin S(x,z), \frac{x+y}{2} \in W, \frac{x+z}{2} \in W, \lVert x-y \rVert \ge b_t,\lVert x-z \rVert \ge b_t\Big\}\\
    &\hspace{2cm}\times \mathbb{P}(\eta_t(S(x,y) \cup S(x,z))=0)\, \dint(x,y,z)\\
    &\le 4t^3 \int_{(\RR^d)^3} \mathbbm{1} \Big\{\frac{x+y}{2} \in W, \frac{x+z}{2} \in W, \lVert x - y \rVert \ge \lVert x - z \rVert\ge b_t \Big\}
		e^{-t \vol(S(x,y) \cup S(x,z))}\, \dint(x,y,z).
\end{align*}
Next we use assumptions (i) and (iii) of Theorem \ref{thm:PoisConvAbstract} to see that
\begin{align*}
    \vol(S(x,y) \cup S(x,z)) &= \vol(S(x,z)) + \vol(S(x,y) \setminus S(x,z))\ge \gamma \lVert x-z \rVert^d + \beta \Big\lVert \frac{y-z}{2} \Big\rVert \lVert x - y \rVert^{d-1},
\end{align*}
whenever $S(x,y) \cap S(x,z)\neq\varnothing$ and $\|x-y\|\geq\|x-z\|$. If the intersection is empty, we have
$$
\vol(S(x,y) \cup S(x,z)) = \vol(S(x,y)) + \vol(S(x,z)) = \gamma \|x-y\|^d + \gamma \|x-z\|^d.
$$
This implies 
\begin{align*}
E_4 & \leq  4t^3 \int_{(\RR^d)^3} \mathbbm{1} \Big\{\frac{x+y}{2} \in W, \frac{x+z}{2} \in W, \lVert x - y \rVert \ge \lVert x - z \rVert\ge b_t \Big\}\\
    &\hspace{2cm}\times e^{-t (\gamma \lVert x-z \rVert^d + \beta \| \frac{y-z}{2} \| \| x - y \|^{d-1})}\, \dint(x,y,z) \\
		& \quad +  4t^3 \int_{(\RR^d)^3} \mathbbm{1} \Big\{\frac{x+y}{2} \in W, \frac{x+z}{2} \in W, \lVert x - y \rVert \ge \lVert x - z \rVert\ge b_t \Big\}\\
    &\hspace{2cm}\times e^{-t \gamma (\| x-y \|^d + \|x-z\|^d)}\, \dint(x,y,z)\\
		&=: E_4' + E_4''.
\end{align*}

Putting $u = x-y$ and $v = x-z$ yields
\begin{align*}
 E_4' & =  4t^3 \int_{(\RR^d)^3} \mathbbm{1} \Big\{ x - \frac{u}{2} \in W, x - \frac{v}{2} \in W, \lVert u \rVert \ge\lVert v \rVert\ge b_t\Big\} e^{-t (\gamma \lVert v \rVert^d + \beta \lVert (v-u)/2 \rVert \lVert u \rVert^{d-1})} \,\dint(x,u,v)\\
&\le 4t^3 \vol(W) \int_{(\RR^d)^2} \mathbbm{1} \{ \lVert u \rVert \ge \lVert v \rVert \ge b_t\}e^{-t (\gamma \lVert v \rVert^d + \beta \lVert (v-u)/2 \rVert \lVert u \rVert^{d-1})} \,\dint(u,v)\\
    &\le 4t^3 \vol(W) \int_{(\RR^d)^2} \mathbbm{1} \{ \lVert v \rVert \ge b_t\} e^{-t (\gamma\lVert v \rVert^d + \beta \lVert (v-u)/2 \rVert \lVert v \rVert^{d-1})} \,\dint(u,v)\\
     &   = 4t^3 \vol(W) \int_{(\RR^d)^2} \mathbbm{1} \{ \lVert v \rVert\ge b_t\} e^{-t (\gamma\lVert v \rVert^d + \beta \lVert w/2 \rVert \lVert v \rVert^{d-1})} \,\dint(v,w)\\
    &= 4t^3 \vol(W) (d\kappa_d)^2 \int_{[0,\infty)^2} \mathbbm{1} \{s \ge b_t\} e^{-t (\gamma s^d + \beta r s^{d-1}/2)} (rs)^{d-1} \,\dint(r,s).
\end{align*}
Arguing as above, we have
\[
e^{-t \beta r s^{d-1}/2} r^{d-1} \le M_{d-1} e^{-t \beta r s^{d-1}/4} \frac{2^{d-1}}{(t\beta s^{d-1})^{d-1}}
\]
so that
\begin{align*}
  E_4'  &\le 4t^3 \vol(W) M_{d-1} (d\kappa_d)^2 \int_{[0,\infty)^2} \mathbbm{1} \{s\ge b_t\}e^{-t\gamma s^d} s^{d-1} e^{-t \beta r s^{d-1}/4} \frac{2^{d-1}}{(t\beta s^{d-1})^{d-1}}\,\dint(r,s)\\
    &= 4t^3 \vol(W) M_{d-1} (d\kappa_d)^2 \int_0^\infty \mathbbm{1} \{s \ge b_t\} e^{-t\gamma s^d} s^{d-1} \frac{2^{d+1}}{(t\beta s^{d-1})^d}\,\dint s\\
    &{\le \frac{2^{d+3}t^{3-d} \vol(W) M_{d-1} (d\kappa_d)^2 }{\beta^db_t^{d(d-1)}} \int_{b_t}^\infty e^{-t\gamma s^d} s^{d-1}\,\dint s.} 
\end{align*}
The last integral has already been computed in \eqref{eqn:integral_b_t} and equals $\frac{2e^{-b}}{t^2d\kappa_d}$. Plugging this in yields
\[
E_4' \le \frac{2^{d+3}t^2 \vol(W) M_{d-1}(d\kappa_d)^2}{\beta^d(t b_t^d)^{(d-1)}} \cdot \frac{2e^{-b}}{t^2d\kappa_d} = \frac{2^{d+4} M_{d-1}d\kappa_d \vol(W)e^{-b}}{{\beta^d((b + \log(t) + \tau)/\gamma)^{d-1}}} = \frac{{C_{4,1}}}{e^b(b+\log(t)+\tau)^{d-1}}.
\]
Using the same substitution as above, spherical coordinates, and \eqref{eqn:integral_b_t} leads to
\begin{align*}
E_4'' & \leq 4t^3 \vol(W) \int_{(\RR^d)^2} \mathbbm{1} \{ \|u\| \geq \lVert v \rVert \ge b_t\} {e^{-t \gamma(\|u\|^d +\| v \|^d)}} \,\dint(u,v) \\
& \leq 4t^3 \vol(W) \bigg( d\kappa_d \int_{b_t}^\infty e^{\gamma t s^d} s^{d-1} \, \dint s \bigg)^2 = \frac{4t^3 \vol(W) (d\kappa_d)^2 (2e^{-b})^2}{(d\kappa_dt^2)^2} = \frac{16\vol(W) e^{-2b}}{t}{=\frac{C_{4,2}}{e^{2b} t}}.
\end{align*}
This completes the proof.
\end{proof}

\subsection{Proofs of Theorem \ref{thm:PoisConvAbstract} and Theorem \ref{thm:GumbelAbstract}}\label{subsec:AbstractProofs}

The proof of Theorem \ref{thm:PoisConvAbstract} now follows immediately by combining the bounds derived in the previous section with the Poisson process approximation result in Theorem \ref{thm:BSY}.

\begin{proof}[Proof of Theorem \ref{thm:PoisConvAbstract}]
{Let $t\geq 2$. We first assume that $t>\frac{2\gamma}{\kappa_d}e^{-b}$,} which implies $b>{\log(\frac{2\gamma}{t\kappa_d})=-\log(t)}-\tau$. {So, plugging in
$$
\mathbf{M}_t|_{W \times [b, \infty)\times\LL^d} = \vol\otimes\mu\otimes\nu|_{W \times [b, \infty)\times\LL^d},
$$
which follows from Lemma \ref{lem:intmeas}, and the estimates provided in Lemma \ref{lem:estimates}} into Theorem \ref{thm:Poisson-Approximation} yields
\begin{align*}
\mathrm{d}_{\mathrm{KR}} (\xi_t|_{W \times [b, \infty)\times\LL^d}, \zeta|_{W \times [b, \infty)\times\LL^d}) &\le  C_2 \frac{1+ b+ \log(t) + \tau}{e^{2b}t} + \frac{C_3}{e^b(b+\log(t)+\tau)^{\max\{d-2, 1/2\}}}\\
 &\quad+ \frac{C_{4,1}}{e^b(b+\log(t)+\tau)^{d-1}} + \frac{C_{4,2}}{e^{2b}t}
\end{align*}
with the constants $C_2$, $C_3$, $C_{4,1}$, and $C_{4,2}$ from Lemma \ref{lem:estimates}. The result of Theorem \ref{thm:PoisConvAbstract} thus follows for {$t> \max\big\{2,\frac{2\gamma}{\kappa_d}e^{-b}\big\}$.} In case that {$t\in\big[2,\frac{2\gamma}{\kappa_d}e^{-b}\big]$,} which yields $b\leq -{\log(t)}-\tau$, the definition of the Kantorovich--Rubinstein distance and Lemma \ref{lem:intmeas} lead to
\begin{align*}
\mathrm{d}_{\mathrm{KR}} (\xi_t|_{W \times [b, \infty)\times\LL^d}, \zeta|_{W \times [b, \infty)\times\LL^d}) & \leq \mathbb{E}[\xi_t(W \times [b, \infty)\times\LL^d)] + \mathbb{E}[\zeta(W \times [b, \infty)\times\LL^d)] \\
& = \vol(W) \mu([-\log(t)-\tau,\infty)) + \vol(W) \mu([b,\infty)) \\
& \leq 2\vol(W) \mu([b,\infty)) = 2\vol(W) e^{-b}.
\end{align*}		
This allows to choose the constant in {\eqref{eqn:bound_dKR}} such that the statement is valid for all $t\geq 2$. {Since convergence in Kantorovich--Rubinstein distance implies convergence in distribution (see \cite[Proposition 2.1]{DST}), this yields $\xi_t|_{W \times [b, \infty)\times\LL^d} \overset{d}{\longrightarrow} \zeta|_{W \times [b, \infty)\times\LL^d}$ as $t\to\infty$. Thus, by \cite[Theorem 23.16]{Kallenberg}, one has $\xi_t\overset{d}{\longrightarrow} \zeta$ as $t\to\infty$, which completes the proof.}
\end{proof}

Similarly, Theorem \ref{thm:GumbelAbstract} follows from Lemma \ref{lem:intmeas}, Lemma \ref{lem:estimates}, and Theorem \ref{thm:Poisson-Approximation}.

\begin{proof}[Proof of Theorem \ref{thm:GumbelAbstract}]
We start by noting that, for any $s\in\RR$,
\begin{align*}
&\PP(\gamma tL_{t,W}^d-\log(t)-\tau\leq s) - \PP(G_W\leq s)=\PP(\xi_t(W\times{(s,\infty)}\times\LL^d)=0) - \PP(Y=0),
\end{align*}
where $Y$ is a Poisson distributed random variable with parameter $\vol(W) e^{-s}$. Thus, by Theorem \ref{thm:Poisson-Approximation}, we have
\begin{align*}
    &|\PP(\gamma tL_{t,W}^d-\log(t)-\tau\leq s) - \PP(G_W\leq s)|\\
    &\leq \min\bigg\{1,\frac{1}{\sqrt{\vol(W) e^{-s}}}\bigg\}|\EE[\xi_t(W\times{(s,\infty)}\times\LL^d)]- \vol(W)e^{-s}| \\
		& \quad+ \min\bigg\{\frac{1}{2},\frac{1}{2\vol(W) e^{-s}}\bigg\}(E_2+E_2+E_4)\\
    &=\min\bigg\{1,\frac{e^{s/2}}{\sqrt{\vol(W)}}\bigg\} |\EE[\xi_t(W\times{(s,\infty)}\times\LL^d)]-\vol(W)e^{-s}| + \min\bigg\{\frac{1}{2},\frac{e^{s}}{2\vol(W)}\bigg\} (E_2+E_2+E_4).
\end{align*}
For $s\geq -\log(t)/4-\tau$, we use Lemma \ref{lem:intmeas} and Lemma \ref{lem:estimates} to conclude that
$$
|\EE[\xi_t(W\times{(s,\infty)}\times\LL^d)]-\vol(W)e^{-s}| = 0
$$
and
\begin{align*}
& E_2 \le C'_2 \vol(W) \frac{1+ s+ \log(t) + \tau}{e^{2s}t}, \quad E_3 \le \frac{C'_3 \vol(W)}{e^s(s+\log(t)+\tau)^{\max\{d-2, 1/2\}}},\\
& E_4 \le \frac{C'_4 \vol(W)}{e^s(s+\log(t)+\tau)^{d-1}} +\frac{16\vol(W)}{e^{2s} t}
\end{align*}
with constants $C'_2,C'_3,C'_4\in(0,\infty)$ only depending on $d$, $\alpha$, $\beta$, and $\gamma$. Hence, {one obtains}
\begin{align*}
    \frac{e^s\,E_2}{\vol(W)} &\leq C'_2\frac{1+ s+ \log(t) + \tau}{e^{s}t} \leq C'_2\Big(\max_{u\geq 0}{u\over e^u}\Big){1\over t}+C'_2{\log(t)+1+\tau\over e^{-\log(t)/4-\tau}\,t}\leq C'_5\Big({1\over t}+{\log(t)\over t^{3/4}}\Big),\\
    \frac{e^s\,E_3}{\vol(W)} &\leq \frac{C'_3}{(s+\log(t)+\tau)^{\max\{d-2, 1/2\}}} \leq {C'_3\over({3\over 4}\log(t))^{\max\{d-2, 1/2\}}},\\
    \frac{e^s\,E_4}{\vol(W)} &\leq \frac{C'_4}{(s+\log(t)+\tau)^{d-1}} +\frac{16}{e^{s} t} \leq {C'_4\over({3\over 4}\log(t))^{d-1}} +\frac{16}{e^{-\tau} t^{3/4}}
\end{align*}
with a suitable constant $C'_5\in(0,\infty)$ depending on the same parameters as $C'_2$. It follows that
\begin{equation}\label{eq:GumbelProof1}
\sup_{s\geq -\log(t)/4-\tau}|\PP(\gamma tL_{t,W}^d-\log(t)-\tau\leq s) - \PP(G_W\leq s)| \leq {C'_6\over {\log^{\max\{d-2, 1/2\}}(t)}}
\end{equation}
for all $t\geq 2$, where $C'_6\in(0,\infty)$ is a constant only depending on $d$, $\alpha$, $\beta$, and $\gamma$.

To deal with the complementary case $s< -\log(t)/4-\tau$ note that
$$
\sup_{s< -\log(t)/4-\tau}\PP(G_W\leq s) = \PP(G_W\leq -\log(t)/4-\tau) = \exp(-\vol(W) e^{\log(t)/4+\tau}) = \exp(-\vol(W) e^\tau\,t^{1/4})
$$
and
\begin{align*}
&\sup_{s< -\log(t)/4-\tau} \PP(\gamma tL_{t,W}^d-\log(t)-\tau\leq s) = \PP\big(\gamma tL_{t,W}^d-\log(t)-\tau\leq -\log(t)/4-\tau\big)\\
&\leq \PP(G_W\leq -\log(t)/4-\tau) 
+ |\PP\big(\gamma tL_{t,W}^d-\log(t)-\tau\leq -\log(t)/4-\tau\big)-\PP(G_W\leq -\log(t)/4-\tau)|\\
&\leq \exp(-\vol(W)e^\tau\,t^{1/4}) + {C'_6\over\log^{\max\{d-2, 1/2\}}(t)}\\
&\leq {C'_7\over\log^{\max\{d-2, 1/2\}}(t)}
\end{align*}
with another constant $C'_7\in(0,\infty)$ depending on $\vol(W)$ and the same parameters as $C'_6$. Combining this with \eqref{eq:GumbelProof1} completes the proof.
\end{proof}

\section{Proofs of Theorem \ref{th:PoisConv} and Theorem \ref{th:gumb}}\label{sec:Proofs}

To prove both Theorem \ref{th:PoisConv} and Theorem \ref{th:gumb}, we need to verify that assumptions (i)--(iv) of Theorem \ref{thm:PoisConvAbstract} are satisfied. The results then follow from Theorem \ref{thm:PoisConvAbstract} and Theorem \ref{thm:GumbelAbstract}. The sets $(S(x,y))_{x,y\in\mathbb{R}^d}$ used to construct the empty region graphs in Examples \ref{ex:Gabriel}--\ref{ex:ConvBody} fall into two classes:
\begin{itemize}
        \item[(S1)] $S(x,y)$ is of the form 
        \begin{equation}\label{eq:SetswithoutRot}
        S(x,y) = \frac{x+y}{2} + \lVert x-y \rVert K
        \end{equation}
        for a compact star-shaped set $K\subset\RR^d$. This is the case for Example \ref{ex:ConvBody}.

        \item[(S2)] $S(x,y)$ is of the form
        $$
        S(x,y) = \frac{x+y}{2} + \lVert x-y \rVert \varrho_{x,y} K
        $$
        with a compact set $K\subset\RR^d$ (satisfying a number of technical assumptions which will be specified later) and $\varrho_{x,y}$ a rotation depending only on $x-y$ in such a way that $\varrho_{x,y}=\varrho_{y,x}$. This is the case for Examples \ref{ex:Gabriel}--\ref{ex:TruncatedSlab}.
\end{itemize}

\subsection{Sets of class (S1)}

If $S(x,y)$ is of the form \eqref{eq:SetswithoutRot} for a compact star-shaped set $K\subset\RR^d$, assumptions (i), (ii), and (iv) are trivially satisfied. To verify assumption (iii), we rely on the for following geometric estimate taken from \cite{Schymura}.

\begin{lemma}\label{lem:geomConv}
Let {$K\in\mathcal{B}(\RR^d)$ be bounded} and let $c>0$ and $x\in\RR^d$ be such that $\|x\|\leq c\,{\rm diam}(K)$, where ${\rm diam}(K):=\sup\{\|u-v\|:u,v\in K\}$ is the diameter of $K$. Then, 
$$
\vol(K\setminus(K+x)) \geq {1\over c+1}{\vol(K)\over {\rm diam}(K)}\,\|x\|.
$$
\end{lemma}
\begin{proof}
We start by writing $\vol((K+x)\triangle K) = \vol(K\setminus(K+x)) + \vol((K+x)\setminus K)$ for the symmetric difference $(K+x)\triangle K$. Moreover, we have $\vol(K\setminus(K+x)) = \vol((K+x)\setminus K)$ and
$$
\vol((K+x)\triangle K) \geq {2\over c+1}{\vol(K)\over {\rm diam}(K)}\,\|x\|
$$
from \cite[Lemma 3.12]{Schymura} and a slight adaption of the argument there. This proves the claim.
\end{proof}

{This estimate is now applied to the sets $(S(x,y))_{x,y\in\mathbb{R}^d}$.}

\begin{lemma}\label{lem:volest}
   Assume that {$S$ is as in \eqref{eq:SetswithoutRot} with a compact and star-shaped set $K\subset\RR^d$} and let $c>0$. Let $x,y,u,v \in \RR^d$ be such that $\lVert x-y \rVert \ge \lVert u-v \rVert$. Then,
\[
\vol(S(x,y)\setminus S(u,v)) \ge {1\over c+1}{\vol(K)\over {\rm diam}(K)} \Big\lVert \frac{x+y}{2} - \frac{u+v}{2} \Big\rVert \lVert x-y \rVert^{d-1},
\]
whenever {$\lVert \frac{x+y}{2} - \frac{u+v}{2} \rVert\leq c\,{\rm diam}(K) \lVert x-y \rVert$.}
\end{lemma}
\begin{proof} Using that $K$ is star-shaped, it follows that
    \begin{align*}
\vol(S(x,y)\setminus S(u,v)) &\ge \vol\Big(\Big(\frac{x+y}{2} + \lVert x-y \rVert K\Big)\setminus \Big(\frac{u+v}{2} + \lVert x-y \rVert K\Big)\Big).
\end{align*}
Defining $L:=\frac{x+y}{2} + \lVert x-y \rVert K$ and $z:= \frac{u+v}{2}-\frac{x+y}{2}$, we can thus write
$$
\vol(S(x,y)\setminus S(u,v)) \ge \vol(L\setminus(L+z)) \geq {1\over c+1}{\vol(L)\over {\rm diam}(L)}\,\|z\|
$$
according to Lemma \ref{lem:geomConv}, whenever $\|z\|\leq c\,{\rm diam}(L)$. Using that ${\rm diam}(L)=\|x-y\|{\rm diam}(K)$ and $\vol(L)=\|x-y\|^d\,\vol(K)$, the result follows.
\end{proof}

{Note that $S(x,y)\cap S(u,v)=\varnothing$ whenever $\lVert \frac{x+y}{2} - \frac{u+v}{2} \rVert\geq {\rm diam}(K) \lVert x-y \rVert$ and $\|x-y\|\geq\|u-v\|$ because then the superset
$$
\bigg(\frac{x+y}{2} + \|x-y\| K \bigg) \cap \bigg(\frac{u+v}{2} + \|x-y\| K \bigg)
$$
of the intersection is empty. Thus, choosing $c=1$ in the previous lemma verifies assumption (iii) of Theorem \ref{thm:PoisConvAbstract} for Example \ref{ex:ConvBody}, and eventually completes the proofs of Theorems \ref{th:PoisConv} and \ref{th:gumb} in that case. }

\subsection{Sets of class (S2)}

Our goal is to verify assumptions (i)--(iv) of Theorem \ref{thm:PoisConvAbstract} for Examples \ref{ex:Gabriel}--\ref{ex:TruncatedSlab}. Clearly, (i) and (ii) are satisfied as the involved sets $K$ are bounded and have positive volume. Moreover, (iv) holds by assumption on the rotations $\varrho_{x,y}$. It remains to verify (iii). For that purpose, we need the following geometric result. 
Throughout this section, we denote the $(d-1)$-dimensional volume of a subset of a $(d-1)$-dimensional affine subspace of $\RR^d$ also by $\vol$.

\begin{proposition}\label{prop:geom}
For $d\in\mathbb{N}$ with $d\ge 2$, $u\in(0,\infty)$, and a continuous symmetric function $R:[-u,u]\to[0,\infty)$ let
$$
K := \Big\{ (x_1,\hdots,x_d)\in \RR^d: x_1\in[-u,u], \sum_{i=2}^{d} x_i^2\leq R(x_1)^2 \Big\}.
$$
Assume that there exist $x_K\in K$ and $r_K>0$ such that $B^d(x_K,r_K)\subset K$. Then one can find a constant $c_K\in(0,\infty)$ only depending on $K$ and $d$ such that
\begin{equation}\label{eq:S1KVolumeLowerBound}
\vol(K\setminus (\varrho K + x)) \geq c_K \|x\|    
\end{equation}
for all rotations $\varrho$ and all $x\in\RR^d$ with $K\cap(\varrho K+x)\neq\varnothing$.
\end{proposition}

\begin{proof}
Let a rotation $\varrho$ and $x\in\RR^d$ be given and let $v_\varrho:=\varrho e_1$ with $e_1:=(1,0,\hdots,0)\in\RR^d$. {Moreover, we assume that $v_\varrho\neq \pm e_1$.} We denote by $E_\varrho$ the linear subspace generated by $e_1$ and $v_\varrho$ and define $a_{\varrho,1}:=\frac{e_1+v_\varrho}{\|e_1+v_\varrho\|}$ and $a_{\varrho,2}:=\frac{e_1-v_\varrho}{\|e_1-v_\varrho\|}$, which are orthogonal. {By the assumed symmetry properties of $K$, $\varrho K\cap E_\varrho$ is the mirror image of $K\cap E_\varrho$ with respect to the lines generated by $a_{1,\varrho}$ and $a_{2,\varrho}$ in $E_\varrho$, which implies that}
\begin{equation}\label{eqn:Equality_varrho_I}
\vol(K \cap(t a_{i,\varrho} + a_{i,\varrho}^\perp)) = \vol(\varrho K \cap(t a_{i,\varrho} + a_{i,\varrho}^\perp))
\end{equation}
for all $t\in\RR$ and $i\in\{1,2\}$. For $d\geq 3$ let $a_{3,\varrho},\hdots,a_{d,\varrho}\in\RR^d$ be such that $a_{1,\varrho},\hdots,a_{d,\varrho}$ form an orthonormal base of $\RR^d$. Then we have
\begin{equation}\label{eqn:Equality_varrho_II}
\vol(K \cap(t a_{i,\varrho} + a_{i,\varrho}^\perp)) = \vol(\varrho K \cap(t a_{i,\varrho} + a_{i,\varrho}^\perp))
\end{equation}
for all $t\in\RR$ and $i\in\{3,\hdots,d\}$ because $ta_{i,\varrho} +a_{i,\varrho}^\perp$ is parallel to the lines generated by $e_1$ and $v_\varrho$, and has the same distance to both of them.

Since $K$ and $\varrho K+x$ have the same volume, we obtain
\begin{equation}\label{eqn:symmetric_difference}
\vol(K\setminus (\varrho K + x)) = \frac{1}{2} \vol(K\triangle (\varrho K + x)).
\end{equation}
For $i\in\{1,\ldots,d\}$ let $H_i:=\{y\in\mathbb{R}^d: \langle a_{i,\varrho},y \rangle \leq \langle a_{i,\varrho},x_K \rangle \}$ be the half-space through $x_K$ whose bounding hyperplane has the normal vector $a_{i,\varrho}$ pointing in the outward direction. Note that
\begin{equation}\label{eqn:symmetric_difference_half-space}
\vol(K {\triangle} (\varrho K + x)) \geq \big| \vol(K\cap H_i) - \vol((\varrho K + x)\cap H_i) \big|.
\end{equation}
It follows from Cavalieri's principle that
\begin{align*}
& \vol(K\cap H_i) - \vol((\varrho K + x)\cap H_i) \\
& = \int_{-\infty}^{\langle a_{i,\varrho},x_K\rangle} \vol(K\cap (ta_{i,\varrho}+a_{i,\varrho}^\perp)) \, \dint t - \int_{-\infty}^{\langle a_{i,\varrho},x_K\rangle} \vol((\varrho K + x)\cap(ta_{i,\varrho}+a_{i,\varrho}^\perp)) \, \dint t.
\end{align*}
For $t\in\RR$ we have
\begin{align*}
\vol((\varrho K + x)\cap (t a_{i,\varrho}+a_{i,\varrho}^\perp)) & = \vol(\varrho K \cap (t a_{i,\varrho}-x+a_{i,\varrho}^\perp)) = \vol(\varrho K \cap ((t-\langle a_{i,\varrho},x\rangle) a_{i,\varrho}+a_{i,\varrho}^\perp)) \\
& = \vol(K \cap ((t-\langle a_{i,\varrho},x\rangle) a_{i,\varrho}+a_{i,\varrho}^\perp)),
\end{align*}
where we used \eqref{eqn:Equality_varrho_I} for $i\in\{1,2\}$ and \eqref{eqn:Equality_varrho_II} for $i\in\{3,\hdots,d\}$ in the last step. This implies that
\begin{align*}
& \vol(K\cap H_i) - \vol((\varrho K + x)\cap H_i) \\
& = \int_{-\infty}^{\langle a_{i,\varrho},x_K\rangle} \vol(K\cap (ta_{i,\varrho}+a_{i,\varrho}^\perp)) \, \dint t - \int_{-\infty}^{\langle a_{i,\varrho},x_K\rangle} \vol(K \cap ((t-\langle a_{i,\varrho},x\rangle) a_{i,\varrho}+a_{i,\varrho}^\perp)) \, \dint t \\
& = \int_{-\infty}^{\langle a_{i,\varrho},x_K\rangle} \vol(K\cap (ta_{i,\varrho}+a_{i,\varrho}^\perp)) \, \dint t - \int_{-\infty}^{\langle a_{i,\varrho},x_K\rangle-\langle a_{i,\varrho},x\rangle} \vol(K \cap (t a_{i,\varrho}+a_{i,\varrho}^\perp)) \, \dint t \\
& = \int_{\langle a_{i,\varrho},x_K\rangle-\langle a_{i,\varrho},x\rangle}^{\langle a_{i,\varrho},x_K\rangle} \vol(K\cap (ta_{i,\varrho}+a_{i,\varrho}^\perp)) \, \dint t.
\end{align*}
The absolute value of the right-hand side can be bounded from below by
$$
\bigg|\int_{\langle a_{i,\varrho},x_K\rangle-\langle a_{i,\varrho},x\rangle}^{\langle a_{i,\varrho},x_K\rangle} \vol(B^d(x_K,r_K)\cap (ta_{i,\varrho}+a_{i,\varrho}^\perp)) \, \dint t\bigg| \geq \kappa_{d-1}(r_K/2)^{d-1} \min\{|\langle a_{i,\varrho},x\rangle|,r_K/2\},
$$
where we used that $B^d(x_K,r_K)\cap (ta_{i,\varrho}+a_{i,\varrho}^\perp)$ contains a $(d-1)$-dimensional ball of radius $r_K/2$ if $t\in[ \langle a_{i,\varrho},x_K\rangle-r_K/2,\langle a_{i,\varrho},x_K\rangle+r_K/2]$. Combining this with \eqref{eqn:symmetric_difference} and \eqref{eqn:symmetric_difference_half-space} yields
$$
\vol(K\setminus (\varrho K + x)) \geq \frac{\kappa_{d-1}r_K^{d-1}}{2^d} \min\{|\langle a_{i,\varrho},x\rangle|,r_K/2\}.
$$
Since $a_{1,\varrho},\hdots,a_{d,\varrho}$ form an orthonormal base, we have $\sum_{i=1}^d \langle a_{i,\varrho},x\rangle^2=\|x\|^2$. Thus, there exists some $i_0\in\{1,\hdots,d\}$ such that 
$$
\frac{1}{\sqrt{d}} \|x\| \leq|\langle a_{i_0,\varrho},x\rangle|\leq \|x\|.
$$
The two previous inequalities lead {to
$$
\vol(K\setminus (\varrho K + x)) \geq \frac{\kappa_{d-1}r_K^{d-1}}{2^d} \min\{\|x\|/\sqrt{d},r_K/2\}.
$$
For rotations $\varrho$ such that $v_\varrho=\pm e_1$ we have $\varrho K= K$ so that we can apply Lemma \ref{lem:geomConv} with $c=\sqrt{d}r_K/(2\operatorname{diam}(K))$. Thus, there exists a constant $c_K'\in(0,\infty)$ such that
\begin{equation}\label{eqn:lower_bound_smaller}
\vol(K\setminus (\varrho K + x)) \geq c_K' \|x\|
\end{equation}
for all $x\in\mathbb{R}^d$ with $\|x\|\le\sqrt{d}r_K/2$ and all rotations $\varrho$.}
So it remains to show the desired inequality for all rotations $\varrho$ and $x\in \mathbb{R}^d$ with $\sqrt{d}r_K/2<\|x\|\leq r_0$, where $r_0$ is given by
$$
r_0:=\sup \{\| y\|: y\in\mathbb{R}^d, (\tilde{\varrho}K+y)\cap K\neq \varnothing \text{ for some rotation }\tilde{\varrho}\}.
$$
Next we note that \eqref{eqn:lower_bound_smaller} especially ensures that
$$
\frac{\vol(K\setminus (\varrho K+x))}{\|x\|}>0
$$
for all rotations $\varrho$ and $x\in\RR^d$ with $\|x\|\geq \sqrt{d}r_K/2$. Since, for a fixed $x\in\RR^d$ and a fixed rotation $\varrho$, the function $(\tilde{x},\tilde{\varrho})\mapsto \mathbbm{1}\{y\in\tilde{\varrho} K +\tilde{x}\}$ is continuous at $(x,\varrho)$ for almost every $y\in\mathbb{R}^d$, rewriting the volume as integral and using the dominated convergence theorem shows that the function
$$
(x,\varrho)\mapsto \frac{\vol(K\setminus (\varrho K+x))}{\|x\|}
$$ 
is continuous for all $x\in\mathbb{R}^d\setminus\{0\}$ and rotations $\varrho$. Since a continuous function attains a minimum on a compact set, there exists a constant $\tilde{c}\in(0,\infty)$ such that
$$
\frac{\vol(K\setminus (\varrho K+x))}{\|x\|}\geq {\tilde{c}}
$$
for all rotations $\varrho$ and all $x\in\RR^d$ such that $\sqrt{d}r_K/2\leq \|x\|\leq r_0$. Together with \eqref{eqn:lower_bound_smaller}, this completes the proof.
\end{proof}

The sets $K$ used in Examples \ref{ex:Gabriel}--\ref{ex:TruncatedSlab} meet the requirements of Proposition \ref{prop:geom} so that they satisfy \eqref{eq:S1KVolumeLowerBound}. We observe that in all these cases, the set $K$ is star-shaped.

What remains to verify is that for sets of class (S2) with star-shaped $K$, property \eqref{eq:S1KVolumeLowerBound} implies assumption (iii) of Theorem \ref{thm:PoisConvAbstract}. Using the star-shapedness of $K$ together with \eqref{eq:S1KVolumeLowerBound}, we obtain, {for $x,y,u,v\in\mathbb{R}^d$ with $\|x-y\|\ge\|u-v\|$,}
\begin{align*}
&\vol\Big(\Big({x+y\over 2}+\|x-y\|\varrho_{x,y}K\Big)\setminus\Big({u+v\over 2}+\|u-v\|\varrho_{u,v}K\Big)\Big)\\
&=\vol\Big(\|x-y\|\varrho_{x,y}K\setminus\Big({u+v\over 2}-{x+y\over 2}+\|u-v\|\varrho_{u,v}K\Big)\Big)\\
&\geq \vol\Big(\|x-y\|\varrho_{x,y}K\setminus\Big({u+v\over 2}-{x+y\over 2}+\|x-y\|\varrho_{u,v}K\Big)\Big)\\
&= \|x-y\|^d\,\vol\Big(K\setminus\Big(\|x-y\|^{-1}\varrho_{x,y}^{-1}\Big({u+v\over 2}-{x+y\over 2}\Big)+\varrho_{x,y}^{-1}\varrho_{u,v}K\Big)\Big)\\
&\geq \|x-y\|^d\,c_K\,\|x-y\|^{-1}\Big\|{u+v\over 2}-{x+y\over 2}\Big\|\\
&= c_K\,\|x-y\|^{d-1}\Big\|{u+v\over 2}-{x+y\over 2}\Big\|
\end{align*}
with the constant $c_K$ from \eqref{eq:S1KVolumeLowerBound}. This completes the proof of Theorems \ref{th:PoisConv} and \ref{th:gumb} for Examples \ref{ex:Gabriel}--\ref{ex:TruncatedSlab}.

\subsection*{Acknowledgement}
{Parts of this paper were written when the authors were participants of the Dual Trimester Program \textit{Synergies Between Modern Probability, Geometric Analysis and Stochastic Geometry} at the Hausdorff Research Institute for Mathematics, Bonn. All support is gratefully acknowledged.

HS was supported by the DFG via SFB 1283 \textit{Taming Uncertainty and Profiting From Randomness and low Regularity in Analysis, Stochastics and Their Applications}.}
MS and CT were supported by the DFG via SPP 2265 \textit{Random Geometric Systems}. CT was also supported by the DFG via SPP 2458 \textit{Combinatorial Synergies}.

\end{document}